\documentclass[a4paper,11pt]{amsart}
\usepackage{mathrsfs}
\usepackage{appendix}
\usepackage{amssymb}
\usepackage{fancyhdr}
\usepackage{charter}
\usepackage{typearea}
\usepackage{color}
\usepackage{amsmath,amstext,amsthm,amscd}
\usepackage{url}

\usepackage[a4paper,top=3cm,bottom=3cm,left=3cm,right=3cm]{geometry}

\makeatletter
      \def\@setcopyright{}
      \def\serieslogo@{}
      \makeatother

\newcommand{\Complex}{\mathbb C}
\newcommand{\Real}{\mathbb R}
\newcommand{\N}{\mathbb N}
\newcommand{\ddbar}{\overline\partial}
\newcommand{\pr}{\partial}
\newcommand{\ol}{\overline}
\newcommand{\Td}{\widetilde}
\newcommand{\norm}[1]{\left\Vert#1\right\Vert}
\newcommand{\abs}[1]{\left\vert#1\right\vert}

\newcommand{\set}[1]{\left\{#1\right\}}
\newcommand{\To}{\rightarrow}

\newcommand{\R}{\mathbb{R}}
\newcommand{\C}{\mathbb{C}}
\newcommand{\map}[3]{#1\colon#2\rightarrow#3}

\theoremstyle{plain}
\newtheorem{theorem}{Theorem}[section]
\newtheorem{lemma}[theorem]{Lemma}
\newtheorem{corollary}[theorem]{Corollary}

\newtheorem{definition}[theorem]{Definition}

\newtheorem{example}[theorem]{Example}
\newtheorem{remark}[theorem]{Remark}
\numberwithin{equation}{section}

\begin{document}
\title[{Szeg\H{o} kernels and equivariant embedding theorems for CR manifolds}]
{Szeg\H{o} kernels and equivariant embedding theorems for CR manifolds}
\author[Hendrik Herrmann]{Hendrik Herrmann}
\address{Faculty of  Mathematics und Natural Sciences, University of Wuppertal, Gau{\ss}stra{\ss}e 20, 42119 Wuppertal, Germany}
\thanks{Hendrik Herrmann was partially supported by the CRC TRR 191: ``Symplectic Structures in Geometry, Algebra and Dynamics'' and the Mathematical Institute, University of Cologne. He would like to thank  the Mathematical Institute, Academia Sinica, and the School of Mathematics and Statistics, Wuhan University, for hospitality, a comfortable accommodation and financial support during his visits in January and March - April, respectively.}
\email{hherrmann@uni-wuppertal.de or post@hendrik-herrmann.de}
\author[Chin-Yu Hsiao]{Chin-Yu Hsiao}
\address{Institute of Mathematics, Academia Sinica and National Center for Theoretical Sciences, Astronomy-Mathematics Building, No. 1, Sec. 4, Roosevelt Road, Taipei 10617, Taiwan}
\thanks{Chin-Yu Hsiao was partially supported by Taiwan Ministry of Science and Technology project 106-2115-M-001-012 and Academia Sinica Career Development
Award. }
\email{chsiao@math.sinica.edu.tw or chinyu.hsiao@gmail.com}
\author[Xiaoshan Li]{Xiaoshan Li}
\address{School of Mathematics
and Statistics, Wuhan University, Wuhan 430072, Hubei, China}
\thanks{Xiaoshan Li was supported by  National Natural Science Foundation of China (Grant No. 11501422).}
\email{xiaoshanli@whu.edu.cn}
\dedicatory{In memory of Professor Louis Boutet de Monvel}

\begin{abstract}
We consider  a compact connected CR manifold with a transversal CR locally free $\Real$-action
endowed with a rigid positive CR line bundle.
We prove that a certain weighted Fourier-Szeg\H{o} kernel of the CR sections
in the high tensor powers admits a full asymptotic expansion and we establish $\Real$-equivariant Kodaira embedding theorem for CR manifolds.
Using similar methods we also establish an analytic proof of an $\Real$-equivariant Boutet de Monvel embedding theorem for strongly pseudoconvex CR manifolds. In particular, we obtain equivariant embedding theorems for irregular Sasakian manifolds.
As applications of our results, we obtain Torus equivariant Kodaira and Boutet de Monvel embedding theorems for CR manifolds and Torus equivariant Kodaira embedding theorem for complex manifolds.
\end{abstract}

\maketitle \tableofcontents

\section{Introduction and statement of the main results}\label{s-gue170805}

Let $(X,T^{1,0}X)$ be a torsion free CR manifold of dimension $2n-1$, $n\geq2$. Then there is a Reeb vector field $T\in C^\infty(X,TX)$ such that the flow of $T$ induces a transversal CR $\Real$-action on $X$. The study of $\Real$-equivariant CR embeddablity for $X$ is closely related to important problems in CR geometry, complex geometry and Mathematical physics. For example, for a compact irregular Sasakian manifold $X$, it is important to know if there is an embedding of $X$ preserving the Reeb vector field and this problem is related to $\Real$-equivariant CR embedding problems for torsion free strongly pseudoconvex CR manifolds. Furthermore, $\Real$-equivariant CR embedding problems are also congeneric to $G$-equivariant Boutet de Monvel and Kodaira embedding problems for CR and complex manifolds.

Suppose that all orbits of the flow of $T$ are compact. Then $X$ admits a transversal CR $S^1$-action $e^{i\theta}$. In~\cite{HHL16} and~\cite{HLM16}, we established $S^1$-equivariant Boutet de Monvel and Kodaira embedding theorems. Let us briefly review the method used in~\cite{HLM16}. Assume that $X$ admits a rigid CR line bundle $L$. For $k\in\mathbb N$, let $\mathcal{H}^0_b(X,L^k)$ denote the space of global smooth CR sections with values in $L^k$. The difficulty of Kodaira embedding problem for $X$ comes from the fact that it is very difficult to understand the large $k$ behavior of $\mathcal{H}^0_b(X,L^k)$ even if $L$ is positive. By using the $S^1$-action $e^{i\theta}$, we consider the spaces
\begin{equation}\label{e-gue170902}\begin{split}
&\mathcal{H}^0_{b,m}(X,L^k)=\set{u\in C^\infty(X,L^k);\, \ddbar_bu=0,\ \ Tu=imu},\\
&\mathcal{H}^0_{b,\leq k\delta}(X,L^k)=\oplus_{m\in\mathbb Z, \abs{m}\leq k\delta}\mathcal{H}^0_{b,m}(X,L^k),\ \ \delta>0.
\end{split}\end{equation}
We proved in~\cite{HLM16} that if $L$ is positive and $\delta>0$ small enough, then
\[\mbox{$d_k:={\rm dim\,}\mathcal{H}^0_{b,\leq k\delta}(X,L^k)\approx k^n$ if $k\gg1$}\]
and a weighted Fourier-Szeg\H{o} kernel for $\mathcal{H}^0_{b,\leq k\delta}(X,L^k)$ admits a full asymptotic expansion. Using that weighted Fourier-Szeg\H{o} kernel asymptotics, we showed in~\cite{HLM16} that the map
\[\begin{split}
\hat\Phi_{k,\delta}: X&\To\Complex\mathbb P^{d_k-1},\\
x&\To[f_1(x),\ldots,f_{d_k}(x)]
\end{split}\]
is an embedding if $k$ is large enough,
where $\set{f_1,\ldots,f_{d_k}}$ is an orthonormal basis for the space $\mathcal{H}^0_{b,\leq k\delta}(X,L^k)$ with respect to some $S^1$-invariant $L^2$ inner product such that
for each $j=1,2,\ldots, d_k$, we have $f_j\in \mathcal{H}^0_{b,m_j}(X,L^k)$, for some $m_j\in\mathbb Z$.
It is clear that the map $\hat\Phi_{k,\delta}$ is $S^1$-equivariant and hence $X$ can be $S^1$-equivariant CR embedded into the projective space.

The method mentioned above does not work when there is an orbit of the flow of $T$ which is non-compact. In that case, $X$ admits a transversal CR $\Real$-action $\eta$.  How to get
$\Real$-equivariant embedding theorems for CR manifolds (and irregular Sasakian manifolds) is an important and difficult problem in CR (and Sasaki Geometry). In this paper, we introduce a new idea and we successfully establish $\Real$-equivariant Kodaira embedding theorem for CR manifolds. For example, we obtain an equivariant Kodaira embedding theorem for irregular Sasakian manifolds. Let us briefly describe our idea. Let $T$
be the infinitesimal generator of the $\Real$-action and assume that $X$ admits a rigid CR line bundle $L$.
Consider the operator
\[-iT: C^\infty(X,L^k)\To C^\infty(X,L^k). \]
Assume that there is an $\Real$-invariant $L^2$ inner product $(\,\cdot\,|\,\cdot\,)_k$ on $C^\infty(X,L^k)$ (if $L$ is positive, we can always find an $\Real$-invariant $L^2$ inner product $(\,\cdot\,|\,\cdot\,)_k$ on $C^\infty(X,L^k)$)
and we extend $-iT$ to $L^2$ space by
\[\begin{split}
&-iT: {\rm Dom\,}(-iT)\subset L^2(X,L^k)\To L^2(X,L^k),\\
&{\rm Dom\,}(-iT)=\set{u\in L^2(X,L^k);\, -iTu\in L^2(X,L^k)}.
\end{split}\]
It is easy to see that $-iT$ is self-adjoint with respect to $(\,\cdot\,|\,\cdot\,)_k$. When $T$ comes from $S^1$-action, we can assume ${\rm Spec\,}(-iT)=\set{m;\, m\in\mathbb Z}$ and every element in ${\rm Spec\,}(-iT)$ is an eigenvalue of $-iT$,  where ${\rm Spec\,}(-iT)$ denotes the spectrum of $-iT$.
When $T$ comes from an $\Real$-action, it is very difficult to understand ${\rm Spec\,}(-iT)$. The key observation in this paper is the following: we show that if there exists a Riemannian metric \(g\) on \(X\), such that the \(\R\)-action acts by isometries with respect to this metric, then the $\Real$-action comes from a torus action on $X$. From this result, we prove that if $L$ is positive then the $\Real$-action comes from a torus action on $X$ and by using the torus action, it is not difficult to show that if $L$ is positive, then ${\rm Spec\,}(-iT)$ is countable and
 any element in ${\rm Spec\,}(-iT)$ is an eigenvalue of $-iT$.

 It was known before that the automorphism group of a compact Sasakian manifold is compact (see \cite{Sch95} and \cite{Coe12}) and therefore that the \(\Real\)-action induced by the Reeb flow comes from a torus action (see for example \cite{Coe11}). Using that result, an \(\Real\)-equivariant embedding result for compact Sasakian manifold (and hence for compact strongly pseudoconvex maniflolds with transversal CR vector fields) with vanishing first cohomology was proven in  \cite{Coe11}.
In this work we use elementary tools from Riemannian geometry to study the \(\Real\)-actions on CR manifolds. It turns out that the strongly pseudoconvexity condition can be replaced by the existence of a rigid positive CR line bundle. Furthermore, we can even drop the compactness condition and find out that there only exists two types of transversal CR \(\Real\)-actions
(see Theorem \ref{thm:mainthm}). This enables the study  of non-compact CR manifolds by analytic methods.

However, in this work we restrict ourselves to the compact case, that is we only need to consider \(\Real\)-actions which comes from a CR torus action.

 Now, assume that $L$ is positive, i.e.~there is an open interval $I\subset\Real$ such that $R^L_x-2s\mathcal{L}_x$ is positive, for every $x\in X$ and every $s\in I$, where $R^L$ denotes the curvature of $L$ and $\mathcal{L}$ denotes the Levi form of $X$ (see Definition~\ref{d-gue150808gI}). For simplicity,  we may assume that $]-\delta,\delta[\subset I$, where $\delta>0$. For every $\alpha\in {\rm Spec\,}(-iT)$, put
$$C^\infty_{\alpha}(X,L^k):=\set{u\in C^\infty(X,L^k);\, -iTu=\alpha u}.$$ As \eqref{e-gue170902},  we define
\[
\mathcal{H}^0_{b,\alpha}(X,L^k):=\set{u\in C^\infty_{\alpha}(X,L^k);\, \ddbar_bu=0}\]
and
\[\mathcal{H}^0_{b,\leq k\delta}(X,L^k):=\bigoplus_{\alpha\in{\rm Spec\,}(-iT), \abs{\alpha}\leq k\delta}\mathcal{H}^0_{b,\alpha}(X,L^k).\]
We can prove that $\mathcal{H}^0_{b,\leq k\delta}(X,L^k)$ is finite dimensional (see Lemma \ref{lem:Hdimfinite}) and hence that it is a closed subspace.
Then, we can modify the method used in~\cite{HLM16} and show that a weighted Fourier-Szeg\H{o} kernel for $\mathcal{H}^0_{b,\leq k\delta}(X,L^k)$ admits a full asymptotic expansion and by using the weighted Fourier-Szeg\H{o} kernel asymptotics, we show that the map
\[\begin{split}
\hat\Phi_{k,\delta}: X&\To\Complex\mathbb P^{d_k-1},\\
x&\To[f_1(x),\ldots,f_{d_k}(x)]
\end{split}\]
is an embedding if $k$ is large enough,
where $\set{f_1,\ldots,f_{d_k}}$ is an orthonormal basis for the space $\mathcal{H}^0_{b,\leq k\delta}(X,L^k)$ such that
for each $j=1,2,\ldots, d_k$, we have $f_j\in \mathcal{H}^0_{b,\alpha}(X,L^k)$, for some $\alpha\in{\rm Spce\,}(-iT)$. It is clear that the map $\hat\Phi_{k,\delta}$ is $\Real$-equivariant. As an application, we obtain equivariant Kodaira embedding theorem for irregular Sasakian manifolds. In particular,  we show that a compact transverse Fano irregular Sasakian manifold can be $\Real$-equivariant CR embedded into projective space. It should be mention that the idea of using CR sections to embed CR manifolds into projective space was introduced by Marinescu~\cite{Ma96} (see also~\cite{Ma16}).

When $X$ is strongly pseudoconvex, the $\Real$-action also comes from a torus action on $X$ and we established $\Real$-equivariant Boutet de Monvel embedding theorem
for $X$ by using our  $S^1$-equivariant Boutet de Monvel embedding theorem~\cite{HHL16}.


We now formulate our main results. We refer the reader to Section~\ref{s:prelim} for some standard notations and terminology used here.
Let $(X,T^{1,0}X)$ be a compact connected CR manifold
of dimension $2n-1$, $n\geqslant2$, endowed with a $\Real$-action $\eta$, $\eta\in\Real$: $\eta: X\to X$, $x\in X\To \eta\circ x\in X$.
Let $T$
be the infinitesimal generator of the $\Real$-action. We assume that the $\Real$-action $\eta$ is transversal CR, that is, $T$ preserves the CR structure $T^{1,0}X$,
and $T$ and $T^{1,0}X\oplus\overline{T^{1,0}X}$ generate the complex tangent bundle to $X$.

Let $(L,h^L)$ be a rigid CR line bundle over $X$, where $h^L$ is a rigid Hermitian metric on $L$.  Let $R^L$ be the curvature of $L$ induced by $h^L$. We say that $(L,h^L)$ is a rigid positive CR  line bundle over $X$ if there is an  open interval $I\subset\Real$ such that $R^L-2s\mathcal{L}$ is positive definite at every point of $X$, for every $s\in I$, where $\mathcal{L}$ denotes the Levi form on $X$. For simplicity, in this work, we always assume that
$]-\delta,\delta[\subset I$,
where $\delta>0$. Let $L^k$ be the $k$-th power of $L$. The Hermitian metric on $L^k$ induced by $h^L$ is denoted by $h^{L^k}$. Let
$\langle\,\cdot\,|\,\cdot\,\rangle$ be the rigid Hermitian metric on $\Complex TX$ induced by $R^L$ such that
$$T^{1,0}X\perp T^{0,1}X,\:\: T\perp (T^{1,0}X\oplus T^{0,1}X),\:\:
\langle\,T\,|\,T\,\rangle=1.$$ We denote by $dv_X$ the volume form induced by $\langle\,\cdot\,|\,\cdot\,\rangle$.
Let $(\,\cdot\,|\,\cdot\,)_k$ be the $L^2$ inner product on $C^\infty(X,L^k)$ induced by $h^{L^k}$ and
$dv_X$. Let $L^2(X,L^k)$ be the completion of $C^\infty(X,L^k)$ with respect
to $(\,\cdot\,|\,\cdot\,)_k$. We extend $(\,\cdot\,|\,\cdot\,)_k$ to $L^2(X,L^k)$.
Consider the operator
\[-iT: C^\infty(X,L^k)\To C^\infty(X,L^k)\]
and we extend $-iT$ to $L^2$ space by
\[\begin{split}
&-iT: {\rm Dom\,}(-iT)\subset L^2(X,L^k)\To L^2(X,L^k),\\
&{\rm Dom\,}(-iT)=\set{u\in L^2(X,L^k);\, -iTu\in L^2(X,L^k)}.
\end{split}\]
By using the fact that the $\Real$-action comes from a torus action on $X$ (see Theorem~\ref{thm:mainthm}), we will show that (see Theorem~\ref{t-gue170817yc} and Theorem~\ref{t-gue170828}) $-iT$ is self-adjoint with respect to $(\,\cdot\,|\,\cdot\,)_k$, ${\rm Spec\,}(-iT)$ is countable
and every element in ${\rm Spec\,}(-iT)$ is an eigenvalue of $-iT$, where ${\rm Spec\,}(-iT)$ denotes the spectrum of $-iT$.

Let $\ddbar_b:\Omega^{0,q}(X,L^k)\To\Omega^{0,q+1}(X,L^k)$
be the tangential Cauchy-Riemann operator with values in $L^k$.
For every $\alpha\in {\rm Spec\,}(-iT)$, put
\begin{equation}\label{e-gue150806I}
C^\infty_{\alpha}(X,L^k):=\set{u\in C^\infty(X,L^k);\, -iTu=\alpha u},
\end{equation}
and
\begin{equation}\label{e-gue150806II}
\mathcal{H}^0_{b,\alpha}(X,L^k):=\set{u\in C^\infty_{\alpha}(X,L^k);\, \ddbar_bu=0}.
\end{equation}
It is easy to see that
for every $\alpha\in {\rm Spec\,}(-iT)$, we have
\begin{equation}\label{e-gue150806III}
{\rm dim\,}\mathcal{H}^{0}_{b,\alpha}(X,L^k)<\infty.
\end{equation}
For $\lambda>0$, put
\begin{equation}\label{e-gue150806IV}
\mathcal{H}^0_{b,\leq\lambda}(X,L^k):=\bigoplus_{\alpha\in{\rm Spec\,}(-iT), \abs{\alpha}\leq\lambda}\mathcal{H}^0_{b,\alpha}(X,L^k).
\end{equation}

For every $\alpha\in {\rm Spec\,}(-iT)$, let $L^2_{\alpha}(X,L^k)\subset L^2(X,L^k)$ be the eigenspace of $-iT$ with eigenvalue $\alpha$. It is easy to see that
$L^2_{\alpha}(X,L^k)$ is the completion of $C^\infty_{\alpha}(X,L^k)$ with respect to $(\,\cdot\,|\,\cdot\,)_k$.
Let
\begin{equation}\label{e-gue150807}
Q^{(0)}_{\alpha,k}:L^2(X,L^k)\To L^2_\alpha(X,L^k)
\end{equation}
be the orthogonal projection with respect to $(\,\cdot\,|\,\cdot\,)_k$.
We have the Fourier decomposition
\[L^2(X,L^k)=\bigoplus_{\alpha\in{\rm Spec\,}(-iT)} L^2_\alpha(X,L^k).\]
We first construct a bounded operator on $L^2(X,L^k)$
by putting a weight on the components of the Fourier
decomposition with the help of a cut-off function.
Fix $\delta>0$ and a function
\begin{equation}\label{e-gue160105}
\tau_\delta\in C^\infty_0((-\delta,\delta)),\:\:
0\leq\tau_\delta\leq1, \:\:\text{$\tau_\delta=1$ on
$\left[-\frac{\delta}{2},\frac{\delta}{2}\:\right]$}.
\end{equation}
Let $F_{k,\delta}:L^2(X,L^k)\To L^2(X,L^k)$
be the bounded operator given by
\begin{equation}\label{e-gue150807I}
\begin{split}
F_{k,\delta}:L^2(X,L^k)&\To L^2(X,L^k),\\
u&\mapsto\sum_{\alpha\in{\rm Spec\,}(-iT)}\tau_\delta\left(\frac{\alpha}{k}\right)Q^{(0)}_{\alpha,k}(u).
\end{split}
\end{equation}
For every $\lambda>0$, we consider the partial Szeg\H{o} projector
\begin{equation}\label{e-gue150806V}
\Pi_{k,\leq\lambda}:L^2(X,L^k)\To \mathcal{H}^0_{b,\leq\lambda}(X,L^k)
\end{equation}
which is the orthogonal projection on the space of $\Real$-equivariant CR functions
of degree less than $\lambda$. Finally, we consider the weighted Fourier-Szeg\H{o}
operator
\begin{equation}\label{e-gue150807II}
P_{k,\delta}:=F_{k,\delta}\circ\Pi_{k,\leq k\delta}\circ F_{k,\delta}:L^2(X,L^k)
\To \mathcal{H}^0_{b,\leq k\delta}(X,L^k).
\end{equation}
The Schwartz kernel of $P_{k,\delta}$ with respect to $dv_X$ is the smooth section
$(x,y)\mapsto P_{k,\delta}(x,y)\in L^k_x\otimes(L^k_y)^*$ satisfying
\begin{equation}\label{sk}
(P_{k,\delta}u)(x)=\int_X P_{k,\delta}(x,y)u(y)\,dv_X(y)\,,\:\:u\in L^2(X,L^k).
\end{equation}

Let $f_j=f_j^{(k)}$, $j=1,\ldots,d_k,$ be an orthonormal basis of
$\mathcal{H}^0_{b,\leq k\delta}(X,L^k)$.
Then
\begin{equation}\label{sk1}
\begin{split}
P_{k,\delta}(x,y)&=\sum_{j=1}^{d_k} (F_{k,\delta}f_j)(x)\otimes
\big((F_{k,\delta}f_j)(y)\big)^*,\\
P_{k,\delta}(x,x)&=\sum_{j=1}^{d_k}\big| (F_{k,\delta}f_j)(x)\big|^2_{h^{L^k}}.
\end{split}
\end{equation}
It should be noticed that the full Szeg\H{o} kernel
$\sum_{j=1}^{d_k}|f_j(x)|^2_{h^{L^k}}$ doesn't admit
an asymptotic expansion in general, hence the necessity
of using the cut-off function $F_{k,\delta}$. In the discussion after Corollary 1.2 in~\cite{HLM16}, we gave an example to show that
the full Szeg\H{o} kernel doesn't admit an asymptotic expansion.
In order to describe the Fourier-Szeg\H{o} kernel $P_{k,\delta}(x,y)$ we will localize $P_{k,\delta}$
with respect to a local rigid CR trivializing section $s$ of $L$ on an open set $D\subset X$.
We define the weight of the metric $h^L$ on $L$ with respect to $s$ to be the function
$\Phi\in C^\infty(D)$ satisfying $\abs{s}^2_{h^L}=e^{-2\Phi}$.
We have an isometry
\begin{equation}\label{e-gue150806VI}
U_{k,s}:L^2(D)\to L^2(D,L^k),\:\: u\longmapsto ue^{k\Phi}s^k,
\end{equation}
 with
inverse $U_{k,s}^{-1}:L^2(D,L^k)\to L^2(D)$, $g\mapsto e^{-k\Phi}s^{-k}g$.
The localization of $P_{k,\delta}$ with respect to the trivializing rigid CR section $s$ is given by
\begin{equation}\label{e-gue150806VII}
P_{k,\delta,s}:L^2_{\mathrm{comp}}(D)\To L^2(D),\:\:
P_{k,\delta,s}= U_{k,s}^{-1}P_{k,\delta}U_{k,s},
\end{equation}
where $L^2_{\mathrm{comp}}(D)$ is the subspace of elements of $L^2(D)$
with compact support in $D$. Let $P_{k,\delta,s}(x,y)\in C^\infty(D\times D)$
be the Schwartz kernel of $P_{k,\delta,s}$ with respect to $dv_X$.
The first main result of this work describes the structure of the localized
Fourier-Szeg\H{o} kernel $P_{k,\delta,s}(x,y)$.
\begin{theorem}\label{t-gue150807}
Let $X$ be a compact CR manifold of dimension $2n-1$, $n\geq2$, with a transversal CR locally free $\Real$-action and
let $L$ be a positive rigid CR line bundle on $X$. With the notations and assumptions above,
consider a point $p\in X$ and a canonical coordinates neighborhood
$(D,x=(x_1,\ldots,x_{2n-1}))$ centered at $p=0$.
Let $s$ be a local rigid CR trivializing section of $L$ on $D$ and set $\abs{s}^2_{h}=e^{-2\Phi}$.
Fix $\delta>0$ small enough and $D_0\Subset D$. Then
\begin{equation}\label{sk2}
P_{k,\delta,s}(x,y)=\int_\R e^{ik\varphi(x,y,t)}g(x,y,t,k)dt+O(k^{-\infty})\:\:
\text{on $D_0\times D_0$},
\end{equation}
where $\varphi\in C^\infty( D\times D\times(-\delta,\delta))$ is a phase function such that
for some constant $c>0$ we have
\begin{equation}\label{e-gue150807b}
\begin{split}
&d_x\varphi(x,y,t)|_{x=y}=-2{\rm Im\,}\ddbar_b\Phi(x)+t\omega_0(x),\ \
d_y\varphi(x,y,t)|_{x=y}=2{\rm Im\,}\ddbar_b\Phi(x)-t\omega_0(x),\\
&{\rm Im\,}\varphi(x,y,t)\geq c|z-w|^2,\ \ (x,y,t)\in D\times D\times(-\delta,\delta), x=(z, x_{2n-1}), y=(w, y_{2n-1}),\\
&\mbox{${\rm Im\,}\varphi(x,y,t)+\abs{\frac{\pr\varphi}{\pr t}(x,y,t)}^2\geq c\abs{x-y}^2$,
$(x,y,t)\in D\times D\times(-\delta,\delta)$},\\
&\mbox{$\varphi(x,y,t)=0$ and $\frac{\pr\varphi}{\pr t}(x,y,t)=0$ if and only if $x=y$},
\end{split}
\end{equation}
and $g(x,y,t,k)\in S^{n}_{{\rm loc\,}}(1;D\times D\times(-\delta,\delta))
\cap C^\infty_0(D\times D\times(-\delta,\delta))$ is a symbol with expansion
\begin{equation}\label{e-gue150807bI}
\begin{split}
&g(x,y,t,k)\sim\sum^\infty_{j=0}g_j(x,y,t)k^{n-j}\text{ in }S^{n}_{{\rm loc\,}}
(1;D\times D\times(-\delta,\delta)), \\
\end{split}\end{equation}
and for
$x\in D_0$ and $|t|<\delta$ we have
\begin{equation}\label{e-gue150807a}
g_0(x,x,t)=(2\pi)^{-n}\abs{\det\bigr(R^L_x-2t\mathcal{L}_x\bigr)}\abs{\tau_\delta(t)}^2,
\end{equation}
where $\omega_0\in C^\infty(X,T^*X)$ is the global real $1$-form of unit length orthogonal
to $T^{*1,0}X\oplus T^{*0,1}X$, see \eqref{e-gue150808I}, $\abs{\det{\bigr(R^L_x-2t\mathcal{L}_x\bigr)}}=\abs{\lambda_1(x,t)\cdots\lambda_{n-1}(x,t)}$,
where $\lambda_j(x,t)$, $j=1,\ldots,n-1$, are the eigenvalues of the Hermitian quadratic
form $R^L_x-2t\mathcal{L}_x$ with respect to $\langle\,\cdot\,|\,\cdot\,\rangle$,
$R^L_x$ and $\mathcal{L}_x$ denote the curvature two form of $L$ and the Levi form
of $X$ respectively (see Definition~\ref{d-gue150808g} and Definition~\ref{d-gue150808}).
\end{theorem}

We refer the reader to Section 2.2 in~\cite{HLM16} for the notations in semi-classical analysis used
in Theorem~\ref{t-gue150807}. For more properties of the phase function $\varphi(x,y,t)$, see Section 3.3 in~\cite{HLM16}. For the meaning of canonical coordinates, we refer the reader to the discussion after Definition~\ref{d-gue150808gI}.

We define now the Equivariant Kodaira map. Consider an open set $D\subset X$ with
\begin{equation}\label{e:eqc}
\bigcup_{\eta\in\Real}\eta(D)\subset D\,,
\end{equation}
and let $s:D\to L$ be a local rigid CR trivializing section on $D$, where
\[\eta(D):=\set{\eta\circ x;\, x\in D}.\]
For any $u\in C^\infty(X,L^k)$
we write $u(x)=s^k(x)\otimes\widetilde u(x)$ on $D$,
with $\widetilde u\in C^\infty(D)$.
Let $\{f_j\}_{j=1}^{d_k}$ be an orthonormal basis of
$\mathcal H^0_{b, \leq k\delta}(X, L^k)$
with respect to $(\,\cdot\,|\,\cdot\,)_k$ such that
$f_j\in\mathcal H^0_{b, \alpha_j}(X, L^k)$, for some $\alpha_j\in{\rm Spec\,}(-iT)$.
Set $g_j=F_{k,\delta}f_j$, $j=1,\ldots,d_k$. The Equivariant Kodaira map is defined on $D$ by
\begin{equation}\label{e-gue150807h0}
\begin{split}
\Phi_{k,\delta}:D&\longrightarrow\mathbb C\mathbb P^{d_k-1},\\
x&\longmapsto\big[F_{k,\delta}f_1(x),\ldots,F_{k,\delta}f_{d_k}(x)\big]:=
\big[\widetilde g_1(x), \ldots, \widetilde g_{d_k}(x)\big],\;\;\text{for $x\in D$}.
\end{split}
\end{equation}
We will show in Theorem~\ref{t-gue170811} and Corollary~\ref{c-gue170814} that there exists
an open cover of $X$ with sets $D$ satisfying \eqref{e:eqc}.
Thus we have a well-defined global map
\begin{equation}\label{e-gue150807h0a}
\Phi_{k,\delta}:X\longrightarrow\mathbb C\mathbb P^{d_k-1},\quad
x\longmapsto\big[F_{k,\delta}f_1(x),\ldots,F_{k,\delta}f_{d_k}(x)\big]=:\big[\Phi^1_{k,\delta}(x),\ldots,\Phi^{d_k}_{k,\delta}(x)\big].
\end{equation}
Since $g_j\in{\mathcal H}^0_{b, \alpha_j}(X, L^k)$ we have
$-iT\,\widetilde g_j=\alpha_j\widetilde g_j$ hence
\[g_j(\eta\circ x)=s^k(\eta\circ x)\otimes\widetilde g_j(\eta\circ x)=
s^k(\eta\circ x)\otimes e^{i\alpha_j\eta}\widetilde g_j(x).\]
Thus
\begin{equation}
\begin{split}
\Phi_{k, \delta}(\eta\circ x)&=[\widetilde g_1(\eta\circ x), \cdots, \widetilde g_{d_k}(\eta\circ x)]=
[e^{\alpha_1\eta}\widetilde g_1(x), \cdots, e^{i\alpha_{d_k}\eta}\widetilde g_{d_k}(x)]\\
&=\big[e^{i\alpha_1\eta}\Phi^1_{k,\delta}(x),\ldots,
e^{i\alpha_k\eta}\Phi^{d_k}_{k,\delta}(x)\big].
\end{split}
\end{equation}
We are thus led to consider \emph{weighted diagonal} $\Real$-actions on
$\mathbb C\mathbb P^N$, that is, actions for which there exists
$(\alpha_1,\ldots,\alpha_{N},\alpha_{N+1})\in\Real^{N+1}$
such that for all $\eta\in\Real$,
\begin{equation}\label{e:equi}
\eta\circ [z_1,\ldots,z_{N+1}]=
\big[e^{i\alpha_1\eta}z_1,\ldots,e^{i\alpha_{N+1}\eta}z_{N+1}\big],\:\:
[z_1,\ldots,z_{N+1}]\in\mathbb C\mathbb P^N.
\end{equation}

\begin{theorem}\label{t-gue150807I}
Let $(X, T^{1,0}X)$ be a compact CR manifold of dimension $2n-1$, $n\geq2$, with a transversal CR $\Real$-action $\eta$.
Assume there is a rigid positive CR line bundle $L$ over $X$.
Then there exists $\delta_0>0$ such that for all $\delta\in(0,\delta_0)$ there exists
$k(\delta)$ so that for $k>k(\delta)$ and any orthonormal basis
$\{f_j\}_{j=1}^{d_k}$ of
$\mathcal H^0_{b, \leq k\delta}(X, L^k)$
with respect to $(\,\cdot\,|\,\cdot\,)_k$ such that
$f_j\in\mathcal H^0_{b, \alpha_j}(X, L^k)$, for some $\alpha_j\in{\rm Spec\,}(-iT)$,
the map $\Phi_{k,\delta}$ introduced in \eqref{e-gue150807h0}
is a smooth CR embedding which is $\Real$-equivariant with respect to
the weighted diagonal $\Real$-action on $\Complex\mathbb P^{d_k-1}$ defined by $(\alpha_1,\ldots,\alpha_{d_k})\in\Real^{d_k}$
as in \eqref{e:equi}, that is,
\[\Phi_{k,\delta}(\eta\circ x)=
\eta\circ\Phi_{k,\delta}(x),\:\:x\in X,\; \eta\in\Real.\]
In particular, the image $\Phi_{k,\delta}(X)\subset\mathbb C\mathbb P^{d_k-1}$ is a CR submanifold
with an induced
weighted diagonal locally free  $\Real$-action.
\end{theorem}

\begin{remark}\label{r-gue170818}
Let $\{f_j\}_{j=1}^{d_k}$ be an orthonormal basis of
$\mathcal H^0_{b, \leq k\delta}(X, L^k)$
with respect to $(\,\cdot\,|\,\cdot\,)_k$ such that
$f_j\in\mathcal H^0_{b, \alpha_j}(X, L^k)$, for some $\alpha_j\in{\rm Spec\,}(-iT)$. As above, we define
\begin{equation}\label{e-gue170818yh}
\begin{split}
\hat\Phi_{k,\delta}:X&\longrightarrow\mathbb C\mathbb P^{d_k-1},\\
x&\longmapsto\big[f_1(x),\ldots,f_{d_k}(x)\big].
\end{split}
\end{equation}
From Theorem~\ref{t-gue150807I}, it is easy to see that $\hat\Phi_{k,\delta}$
is a smooth CR embedding which is $\Real$-equivariant with respect to
the weighted diagonal $\Real$-action on $\Complex\mathbb P^{d_k-1}$ defined by $(\alpha_1,\ldots,\alpha_{d_k})\in\Real^{d_k}$
as in \eqref{e:equi}.
\end{remark}

Ohsawa and Sibony~\cite{Oh12,OS00}
constructed for every $\kappa\in\mathbb N$ a CR projective embedding of class $C^\kappa$
of a Levi-flat CR manifold by using $\ol\partial$-estimates.
The second author and Marinescu~\cite{HM15} gave a Szeg\H{o} kernel proof of Ohsawa and Sibony's result.
A natural question is whether we can improve the regularity to $\kappa=\infty$.
Adachi~\cite{Ad13} showed that the answer is no, in general.
The analytic difficulty of this problem comes from the fact that
the Kohn Laplacian is not hypoelliptic on Levi flat manifolds.
The second, third author and Marinescu \cite{HLM16} generalized Ohsawa and Sibony's result to $C^\infty$-smooth when the CR manifold admits a transversal CR circle action and the CR line bundle is rigid and positive. Theorem \ref{t-gue150807I} above shows that the circle action in \cite {HLM16} can be weakened to  a $\mathbb R$-action.

\begin{corollary}\label{c-gue170818a}
Let $(X, T^{1,0}X)$ be a compact irregular Sasakian manifold and let $T$ be the associated Reeb vector field. Let $\eta\in\Real$ be the $\Real$-action induced by the Reeb vector field. If $X$ admits a rigid positive CR line bundle $L$, then, for \(k\) sufficiently large, the maps $\Phi_{k,\delta}$ and $\hat\Phi_{k,\delta}$ are smooth CR embeddings  of $X$ in $\mathbb C\mathbb P^{d_k-1}$
which are $\Real$-equivariant with respect to a weighted diagonal $\Real$-actions on $\Complex\mathbb P^{d_k-1}$ (cf.\ Theorem \ref{t-gue150807I} and Remark~\ref{r-gue170818}).
\end{corollary}

\begin{corollary}\label{c-gue170818aI}
Let $(X, T^{1,0}X)$ be a compact transverse Fano irregular Sasakian manifold. Let $T$ be the associated Reeb vector field  on $X$ and let $\eta\in\Real$ be the $\Real$-action induced by the Reeb vector field. Then , for \(k\) sufficiently large, the maps $\Phi_{k,\delta}$ and $\hat\Phi_{k,\delta}$ are smooth CR embeddings of $X$ in $\mathbb C\mathbb P^{d_k-1}$
which are $\Real$-equivariant with respect to a weighted diagonal $\Real$-actions on $\Complex\mathbb P^{d_k-1}$ (cf.\ Theorem \ref{t-gue150807I} and Remark~\ref{r-gue170818}).
\end{corollary}

\begin{remark}
	If the rigid CR line bundle over the Sasakian manifold $(X, T^{1,0}X)$ in Corollary \ref{c-gue170818a} or Corollary \ref{c-gue170818aI} is not positive the maps $\Phi_{k,\delta}$ and $\hat\Phi_{k,\delta}$ may fail to be embeddings.
	However, in that case we can still find a smooth $\Real$-equivariant embedding of \(X\) into some \(\C\mathbb P^N\)  with respect to a weighted diagonal $\Real$-actions on $\Complex\mathbb P^{N}$. That follows from Corollary \ref{c-gue170817} using the map \(\C^N\ni(z_1,\ldots,z_N)\mapsto[1,z_1,\ldots,z_N]\in \Complex\mathbb P^{N}\).
\end{remark}

Now, we consider a torsion free compact connected strongly pseudoconvex CR manifold $(X,T^{1,0}X)$ of dimension $2n-1$, $n\geq2$. Then, $X$ admits a transversal CR $\Real$-action $\eta$, $\eta\in\Real$: $\eta: X\to X$, $x\mapsto \eta\circ x$.
By using the fact that the $\Real$-action comes from a torus action on $X$ (see Corollary~\ref{cor:psc}) and the equivariant embedding theorem established in~\cite{HHL16}, we get (see the proof of Theorem~\ref{thm:equivariantembeddingpseudoconvex})

\begin{theorem}\label{t-gue170817}
		Let  $X$ be a connected compact strongly pseudoconvex CR manifold equipped with a transversal CR $\Real$-action $\eta$.
		Then, there exists $N\in\mathbb N$, $\nu_1,\ldots,\nu_N\in\Real$ and a CR embedding $\map{\Phi=(\Phi_1,\ldots,\Phi_N)}{X}{\Complex^N}$
		such that \[\Phi(\eta\circ x)=(e^{i\nu_1\eta}\Phi_1(x),\ldots,e^{i\nu_N\eta}\Phi_N(x))\] holds for all $x\in X$ and $\eta\in\Real$. In other words, $\Phi$ is eqivariant with respect to the holomorphic $\Real$-action $\eta\circ z=(e^{i\nu_1\eta}z_1,\ldots,e^{i\nu_N\eta}z_N)$ on $\C^N$.
	\end{theorem}
	
	\begin{corollary}\label{c-gue170817}
		Let \(X\) be a compact connected Sasakian manifold with Reeb vector field \(T\). There exist \(N\in\N\), \(\nu_1,\ldots,\nu_N\in\R\) and an equivariant embedding \(\map{\Phi}{X}{\C^N}\) with respect to the \(\R\)-action on \(X\) generated by \(T\) and the \(\R\)-action \(\eta\circ z=(e^{i\nu_1\eta}z_1,\ldots,e^{i\nu_N\eta}z_N)\) on \(\C^N\). Furthermore, we have
		\[\Phi_* T_x=i\sum_{j=1}^N\nu_j(z_j\frac{\partial}{\partial z_j}-\overline{z_j}\frac{\partial}{\partial \overline{z_j}})\left|_{z=\Phi(x)}\right..\]
	\end{corollary}	

\subsection{Applications: Torus equivariant Kodaira and Boutet de Monvel embedding theorems for CR manifolds}\label{s-gue170818t}

We will apply Theorem~\ref{t-gue150807I} to establish torus equivariant Kodaira embedding theorem for CR manifolds.
Let $(X, T^{1,0}X)$ be a compact connected orientable CR manifold of dimension $2n-1$, $n\geq2$.
We assume that $X$ admits a $d$-dimensional Torus  action \(T^d\curvearrowright X\) denoted by $(e^{i\theta_1},\ldots,e^{i\theta_d})$. Let $\mathfrak{g}$ denote the Lie algebra of $T^d$. For any $\xi \in \mathfrak{g}$, we write $\xi_X$ to denote the vector field on $X$ induced by $\xi$, that is, $(\xi_X u)(x)=\frac{\partial}{\partial t}\left(u(\exp(t\xi)\circ x)\right)|_{t=0}$, for any $u\in C^\infty(X)$. Let $\underline{\mathfrak{g}}={\rm Span\,}\set{\xi_X;\, \xi\in\mathfrak{g}}$. For every $j=1,\ldots,d$, let $T_j$ be the vector field on $X$ given by
\[(T_ju)(x)=\frac{\pr}{\pr\theta_j} u((1,\ldots,1,e^{i\theta_j},1,\ldots,1)\circ x)|_{\theta_j=0}.\]
We have $\underline{\mathfrak{g}}={\rm span\,}\set{T_j;\, j=1,\ldots,d}$. We assume that
\begin{equation}\label{e-gue170808}
\begin{split}
&[T_j, C^\infty(X,T^{1,0}X)]\subset C^\infty(X,T^{1,0}X),\ \ j=1,2,\ldots,d,\\
&{\rm span\,}\set{T^{1,0}X\oplus T^{0,1}X, \Complex\underline{\mathfrak{g}}}=\Complex TX.
\end{split}
\end{equation}
Suppose that $X$ admits a torus invariant CR line bundle $L$. For every $(m_1,\ldots,m_d)\in\mathbb Z^d$, put
\begin{equation}\label{e-gue170819m}
\begin{split}
&\mathcal{H}^0_{b,m_1,\ldots,m_d}(X,L^k)\\
&=\{u\in C^\infty(X,L^k);\, \ddbar_bu=0,  u((e^{i\theta_1},\ldots,e^{i\theta_d})\circ x)=e^{im_1\theta_1+\cdots+im_d\theta_d}u(x),\\
&\quad\forall x\in X,\ \ \forall (e^{i\theta_1},\ldots,e^{i\theta_d})\in T^d\}.\end{split}\end{equation}

\begin{theorem}\label{t-gue170808}
With the notations and assumptions used above, assume that $L$ admits a $T^d$-invariant Hermitian metric $h^L$ such that the induced curvature $R^L$ is positive.
Fix any $\beta_j\in\Real$, $j=1,\ldots,d$, where $\beta_j$, $j=1,\ldots,d$, are linear independent over $\mathbb Q$.
Then there is a $k_0>0$ such that for all $k\geq k_0$,
there is a torus equivariant CR embedding
\[\begin{split}
\Phi_k: X&\To \Complex\mathbb P^{N_k-1},\\
x&\mapsto[g_1(x),\ldots,g_{N_k}(x)],
\end{split}\]
such that $g_j\in\mathcal{H}^0_{b,m_{j_1},\ldots,m_{j_d}}(X,L^k)$, for some $(m_{j_1},\ldots, m_{j_d})\in\mathbb Z^d$ with
\[\abs{m_{j_1}\beta_1+\cdots+m_{j_d}\beta_d}\leq k\delta,\]
$j=1,\ldots,N_k$, where $N_k\in\mathbb N$.
\end{theorem}

\begin{proof}
From \eqref{e-gue170808}, there is a real non-vanishing vector field $T\in\underline{\mathfrak{g}}$ such that
\[T^{1,0}X\oplus T^{0,1}X\oplus\Complex T=\Complex TX.\]
By continuity, we may take $T=\beta_1T_1+\cdots+\beta_dT_d$. Then, $X$ admits a locally free $\Real$-action $\eta$:
\[\begin{split}
\eta: X&\To X,\\
x&\mapsto \eta\circ x=(e^{i\beta_1\eta},\ldots,e^{i\beta_d\eta})\circ x,
\end{split}\]
 and $T$ is the infinitesimal generator of the $\Real$-action $\eta$. From \eqref{e-gue170808}, we see that the $\Real$-action is transversal and CR.
 Take a $T^d$-invariant Hermitian metric $h^L$ on $L$ such that the induced curvature $R^L$ is positive. Then, $h^L$ is also $T$-rigid. As before, let
$\langle\,\cdot\,|\,\cdot\,\rangle$ be the rigid Hermitian metrix on $\Complex TX$ induced by $R^L$ such that
$$T^{1,0}X\perp T^{0,1}X,\:\: T\perp (T^{1,0}X\oplus T^{0,1}X),\:\:
\langle\,T\,|\,T\,\rangle=1$$ and let $(\,\cdot\,|\,\cdot\,)_k$ be the $L^2$ inner product on $C^\infty(X,L^k)$ induced by $h^{L^k}$ and $\langle\,\cdot\,|\,\cdot\,\rangle$.
For every $(m_1,\ldots,m_d)\in\mathbb Z^d$, put
\[\begin{split}
&C^\infty_{m_1,\ldots,m_d}(X,L^k):=\{u\in C^\infty(X,L^k);\, u((e^{i\theta_1},\ldots,e^{i\theta_d})\circ x)=e^{im_1\theta_1+\cdots+im_d\theta_d}u(x),\\
&\quad\forall x\in X,\ \ \forall (e^{i\theta_1},\ldots,e^{i\theta_d})\in T^d\}.\end{split}\]
For $u\in C^\infty(X)$, we have the orthogonal decomposition with respect to $(\,\cdot\,|\,\cdot\,)_k$:
 \begin{equation}\label{e-gue170809yI}
 u(x)=\sum_{(m_1,\ldots,m_d)\in\mathbb Z^d}u_{m_1,\ldots,m_d}(x),\ \ u_{m_1,\ldots,m_d}(x)\in C^\infty_{m_1,\ldots,m_d}(X,L^k).
 \end{equation}
 From \eqref{e-gue170809yI} and note that $\beta_j$, $j=1,\ldots,d$, are linear independent over $\mathbb Q$, it is easy to check that
 \begin{equation}\label{e-gue170809y}
 C^\infty_{m_1,\ldots,m_d}(X,L^k)=C^\infty_\alpha(X,L^k),\ \ \alpha=\beta_1m_1+\cdots+\beta_dm_d\in{\rm Spec\,}(-iT),
 \end{equation}
 where $C^\infty_\alpha(X,L^k)$ is as in \eqref{e-gue150806I}. From \eqref{e-gue170809y}, we conclude that for every $\alpha\in {\rm Spec\,}(-iT)$, we have
\begin{equation}\label{e-gue170809}
\mathcal{H}^0_{b,\alpha}(X,L^k)=\mathcal{H}^0_{b,m_1,\ldots,m_d}(X,L^k),\ \ \alpha=m_1\beta_1+\cdots+m_d\beta_d,
\end{equation}
where $\mathcal{H}^0_{b,\alpha}(X,L^k)$ is as in \eqref{e-gue150806II}. By Theorem~\ref{t-gue150807I}, we can find an orthonormal basis
$\{g_j\}_{j=1}^{d_k}$ of
$\mathcal H^0_{b, \leq k\delta}(X, L^k)$
with respect to $(\,\cdot\,|\,\cdot\,)_k$ such that
$g_j\in\mathcal H^0_{b, \alpha_j}(X, L^k)$, for some $\alpha_j\in{\rm Spec\,}(-iT)$,
the map $\Phi_{k,\delta}$ introduced in \eqref{e-gue150807h0}
is a smooth CR embedding, where $\mathcal H^0_{b, \leq k\delta}(X, L^k)$ is as in \eqref{e-gue150806IV}.
From \eqref{e-gue170809}, we see that each $g_j$ is in $\mathcal{H}^0_{b,m_{j_1},\ldots,m_{j_d}}(X,L^k)$, for some $(m_{j_1},\ldots,m_{j_d)}\in\mathbb Z^d$ with
$\abs{m_{j_1}\beta_1+\cdots+m_{j_d}\beta_d}\leq k\delta$.
The theorem follows.
\end{proof}

Let $(X, T^{1,0}X)$ be a compact connected  strongly pseudoconvex CR manifold of dimension $2n-1$, $n\geq2$.
We assume that $X$ admits a $d$-dimensional torus action $(e^{i\theta_1},\ldots,e^{i\theta_d})$. Assume that this torus action satisfies \eqref{e-gue170808}.
For every $(m_1,\ldots,m_d)\in\mathbb Z^d$, put
\[\begin{split}
&\mathcal{H}^0_{b,m_1,\ldots,m_d}(X)\\
&=\{u\in C^\infty(X);\, \ddbar_bu=0,  u((e^{i\theta_1},\ldots,e^{i\theta_d})\circ x)=e^{im_1\theta_1+\cdots+im_d\theta_d}u(x),\\
&\quad\forall x\in X,\ \ \forall (e^{i\theta_1},\ldots,e^{i\theta_d})\in T^d\}.\end{split}\]
From Theorem~\ref{t-gue170817} and by repeating the proof of Theorem~\ref{t-gue170808} with minor change, we obtain torus equivariant Boutet de Monvel embedding theorem for strong pseudoconvex CR manifolds.

\begin{theorem}\label{t-gue170817y}
With the assumptions and notations used above,
there is a torus equivariant CR embedding
\[\begin{split}
\Phi: X&\To\mathbb C^{N},\\
x&\mapsto(g_1(x),\ldots,g_N(x)),
\end{split}\]
such that $g_j\in\mathcal{H}^0_{b,m_{j_1},\ldots,m_{j_d}}(X)$, for some $(m_{j_1},\ldots, m_{j_d})\in\mathbb Z^d$, $j=1,\ldots,N$.
\end{theorem}

\subsection{Applications: Torus equivariant Kodaira embedding theorem for complex manifolds}\label{s-gue170819}

Let $(E,h^E)$ be a holomorphic line bundle over a connected compact complex manifold $(M,J)$ with ${\rm dim\,}_{\Complex}M=n$, where $J$ denotes the complex structure map of $M$ and $h^E$ is a Hermitian fiber metric of $E$. 
Assume that $(M,J)$ admits a holomorphic d-dimensional torus action \(T^d\curvearrowright X\) denoted by $(e^{i\theta_1},\ldots,e^{i\theta_d})$ and that the action lifts to a holomorphic action on $E$.
For $(m_1,\ldots,m_d)\in\mathbb Z^d$, put
\begin{equation}\label{e-gue170819mp}
\begin{split}
&\mathcal{H}^0_{m_1,\ldots,m_d}(M,E^k)\\
&=\{u\in C^\infty(M,E^k);\, \ddbar u=0,  u((e^{i\theta_1},\ldots,e^{i\theta_d})\circ x)=e^{im_1\theta_1+\cdots+im_d\theta_d}u(x),\\
&\quad\forall x\in M,\ \ \forall (e^{i\theta_1},\ldots,e^{i\theta_d})\in T^d\}.\end{split}\end{equation}
We have the following Torus equivariant Kodaira embedding theorem

\begin{theorem}\label{t-gue170819mp}
With the notations and assumptions used above, assume that $E$ admits a $T^d$-invariant Hermitian metric $h^E$ such that the induced curvature $R^E$ is positive.
Fix any $\beta_j\in\Real$, $j=1,\ldots,d$, where $\beta_j$, $j=1,\ldots,d$, are linear independent over $\mathbb Q$.
Then there is a $k_0>0$ such that for all $k\geq k_0$,
there is a torus equivariant holomorphic embedding
\[\begin{split}
\phi_k: M&\To \Complex\mathbb P^{N_k-1},\\
x&\mapsto[q_1(x),\ldots,q_{N_k}(x)],
\end{split}\]
such that $q_j\in\mathcal{H}^0_{m_{j_1},\ldots,m_{j_d}}(M,E^k)$ with
\[\abs{m_{j_1}\beta_1+\cdots+m_{j_d}\beta_d+m_{j_{d+1}}}\leq k\delta,\]
for some $(m_{j_1},\ldots, m_{j_d},m_{j_{d+1}})\in\mathbb Z^{d+1}$, $j=1,\ldots,N_k$, where $N_k\in\mathbb N$.
\end{theorem}

\begin{proof}
We will use the same notations as in the proof of Theorem~\ref{t-gue170808}.
Consider $X:=M\times S^1$. Then $X$ is a compact connected CR manifold with CR structure $T^{1,0}_{(x,e^{iu})}X:=T^{1,0}_xM$, for every $(x,e^{iu})\in M\times S^1$. Then $X$ admits a \(T^{d+1}\)-action $(e^{i\theta_1},\ldots,e^{i\theta_d},e^{i\theta_{d+1}})$:
\[(e^{i\theta_1},\ldots,e^{i\theta_d},e^{i\theta_{d+1}})\circ (x,e^{iu}):=((e^{i\theta_1},\ldots,e^{i\theta_d})\circ x, e^{i\theta_{d+1}+iu}),\  \ \forall(x,e^{iu})\in M\times S^1.\]
It is clear that $E$ is a $T^{d+1}$-invariant CR line bundle over $X$ and the \(T^{d+1}\)-action satisfies \eqref{e-gue170808}.
By Theorem~\ref{t-gue170808}, there is a $k_0>0$ such that for all $k\geq k_0$, we can find a CR embedding the map
\begin{equation}\label{e-gue170819mpy}
x\in X\longmapsto\big[f_1(x),\ldots,f_{N_k}(x)\big]\in\mathbb C\mathbb P^{N_k-1}
\end{equation}
such that each $f_j$ is in $\mathcal{H}^0_{b,m_{j_1},\ldots,m_{j_{d+1}}}(X,L^k)$  with $\abs{m_{j_1}\beta_1+\cdots+m_{j_d}\beta_d+m_{j_{d+1}}}\leq k\delta$, for some $(m_{j_1},\ldots, m_{j_d},m_{j_{d+1}})\in\mathbb Z^{d+1}$, $j=1,\ldots,N_k$, where $N_k\in\mathbb N$.
For every $j=1,\ldots,N_k$, let $q_j(x):=f_j(x,e^{iu})|_{u=0}$. Then, $q_j(x)\in\mathcal H^0_{m_1,\ldots,m_d}(M,E^k)$, for some $(m_1,\ldots,m_d)\in\mathbb Z^d$.
It is not difficult  to check that
the map
\[
x\in M\longmapsto\big[q_1(x),\ldots,q_{N_k}(x)\big]\in\mathbb C\mathbb P^{N_k-1}
\]
is a holomorphic embedding. The theorem follows.
\end{proof}

\section{Preliminaries}\label{s:prelim}

\subsection{Some standard notations}\label{s-gue150508b}
We use the following notations: $\mathbb N=\set{1,2,\ldots}$,
$\mathbb N_0=\mathbb N\cup\set{0}$, $\Real$
is the set of real numbers, $\ol\Real_+:=\set{x\in\Real;\, x\geq0}$.
For a multiindex $\alpha=(\alpha_1,\ldots,\alpha_m)\in\mathbb N_0^m$
we set $\abs{\alpha}=\alpha_1+\cdots+\alpha_m$. For $x=(x_1,\ldots,x_m)\in\Real^m$ we write
\[
\begin{split}
&x^\alpha=x_1^{\alpha_1}\ldots x^{\alpha_m}_m,\quad
 \pr_{x_j}=\frac{\pr}{\pr x_j}\,,\quad
\pr^\alpha_x=\pr^{\alpha_1}_{x_1}\ldots\pr^{\alpha_m}_{x_m}=\frac{\pr^{\abs{\alpha}}}{\pr x^\alpha}\,,\\
&D_{x_j}=\frac{1}{i}\pr_{x_j}\,,\quad D^\alpha_x=D^{\alpha_1}_{x_1}\ldots D^{\alpha_m}_{x_m}\,,
\quad D_x=\frac{1}{i}\pr_x\,.
\end{split}
\]
Let $z=(z_1,\ldots,z_m)$, $z_j=x_{2j-1}+ix_{2j}$, $j=1,\ldots,m$, be coordinates of $\Complex^m$,
where
$x=(x_1,\ldots,x_{2m})\in\Real^{2m}$ are coordinates in $\Real^{2m}$.
Throughout the paper we also use the notation
$w=(w_1,\ldots,w_m)\in\Complex^m$, $w_j=y_{2j-1}+iy_{2j}$, $j=1,\ldots,m$, where
$y=(y_1,\ldots,y_{2m})\in\Real^{2m}$.
We write
\[
\begin{split}
&z^\alpha=z_1^{\alpha_1}\ldots z^{\alpha_m}_m\,,\quad\ol z^\alpha=\ol z_1^{\alpha_1}\ldots\ol z^{\alpha_m}_m\,,\\
&\pr_{z_j}=\frac{\pr}{\pr z_j}=
\frac{1}{2}\Big(\frac{\pr}{\pr x_{2j-1}}-i\frac{\pr}{\pr x_{2j}}\Big)\,,\quad\pr_{\ol z_j}=
\frac{\pr}{\pr\ol z_j}=\frac{1}{2}\Big(\frac{\pr}{\pr x_{2j-1}}+i\frac{\pr}{\pr x_{2j}}\Big),\\
&\pr^\alpha_z=\pr^{\alpha_1}_{z_1}\ldots\pr^{\alpha_m}_{z_m}=\frac{\pr^{\abs{\alpha}}}{\pr z^\alpha}\,,\quad
\pr^\alpha_{\ol z}=\pr^{\alpha_1}_{\ol z_1}\ldots\pr^{\alpha_m}_{\ol z_m}=
\frac{\pr^{\abs{\alpha}}}{\pr\ol z^\alpha}\,.
\end{split}
\]

Let $X$ be a $C^\infty$ orientable paracompact manifold.
We let $TX$ and $T^*X$ denote the tangent bundle of $X$ and the cotangent bundle of $X$ respectively.
The complexified tangent bundle of $X$ and the complexified cotangent bundle of $X$
will be denoted by $\Complex TX$ and $\Complex T^*X$ respectively. We write $\langle\,\cdot\,,\cdot\,\rangle$
to denote the pointwise duality between $TX$ and $T^*X$.
We extend $\langle\,\cdot\,,\cdot\,\rangle$ bilinearly to $\Complex TX\times\Complex T^*X$.

Let $E$ be a $C^\infty$ vector bundle over $X$. The fiber of $E$ at $x\in X$ will be denoted by $E_x$.
Let $F$ be another vector bundle over $X$. We write
$F\boxtimes E^*$ to denote the vector bundle over $X\times X$ with fiber over $(x, y)\in X\times X$
consisting of the linear maps from $E_y$ to $F_x$.

Let $Y\subset X$ be an open set. The spaces of
smooth sections of $E$ over $Y$ and distribution sections of $E$ over $Y$ will be denoted by $C^\infty(Y, E)$ and $\mathscr D'(Y, E)$ respectively.
Let $\mathscr E'(Y, E)$ be the subspace of $\mathscr D'(Y, E)$ whose elements have compact support in $Y$.
For $m\in\Real$, we let $H^m(Y, E)$ denote the Sobolev space
of order $m$ of sections of $E$ over $Y$. Put
\begin{gather*}
H^m_{\rm loc\,}(Y, E)=\big\{u\in\mathscr D'(Y, E);\, \varphi u\in H^m(Y, E),
      \,\forall\varphi\in C^\infty_0(Y)\big\}\,,\\
      H^m_{\rm comp\,}(Y, E)=H^m_{\rm loc}(Y, E)\cap\mathscr E'(Y, E)\,.
\end{gather*}

\subsection{CR manifolds with $\Real$-action} \label{s-gue150808}

Let $(X, T^{1,0}X)$ be a compact CR manifold of dimension $2n-1$, $n\geq 2$, where $T^{1,0}X$ is a CR structure of $X$. That is $T^{1,0}X$ is a subbundle of rank $n-1$ of the complexified tangent bundle $\mathbb{C}TX$, satisfying $T^{1,0}X\cap T^{0,1}X=\{0\}$, where $T^{0,1}X=\overline{T^{1,0}X}$, and $[\mathcal V,\mathcal V]\subset\mathcal V$, where $\mathcal V=C^\infty(X, T^{1,0}X)$. We assume that $X$ admits a $\Real$-action $\eta$, $\eta\in\Real$: $\eta: X\to X$, $x\mapsto\eta\circ x$.  Let $T\in C^\infty(X, TX)$ be the infinitesimal generator of   the $\Real$-action which  is given by
\begin{equation}\label{e-gue150808}
(Tu)(x)=\frac{\partial}{\partial \eta}\left(u(\eta\circ x)\right)|_{\eta=0},\ \ u\in C^\infty(X).
\end{equation}

\begin{definition}
We say that the $\Real$-action $\eta$ is CR if
$[T, C^\infty(X, T^{1,0}X)]\subset C^\infty(X, T^{1,0}X)$ and the $\Real$-action is transversal if for each $x\in X$,
$\Complex T(x)\oplus T_x^{1,0}(X)\oplus T_x^{0,1}X=\mathbb CT_xX$. Moreover, we say that the $\Real$-action is locally free if $T\neq0$ everywhere.
\end{definition}

Assume that $(X, T^{1,0}X)$ is a compact CR manifold of dimension $2n-1$, $n\geq 2$, with a transversal CR $\Real$-action $\eta$ and we let $T$ be the global vector field induced by the $\Real$-action. Let $\omega_0\in C^\infty(X,T^*X)$ be the global real one form determined by
\begin{equation}\label{e-gue150808I}
\begin{split}
&\langle\,\omega_0\,,\,u\,\rangle=0,\ \ \forall u\in T^{1,0}X\oplus T^{0,1}X,\\
&\langle\,\omega_0\,,\,T\,\rangle=-1.
\end{split}
\end{equation}

\begin{definition}\label{d-gue150808}
For $p\in X$, the Levi form $\mathcal L_p$ is the Hermitian quadratic form on $T^{1,0}_pX$ given by
$\mathcal{L}_p(U,\ol V)=-\frac{1}{2i}\langle\,d\omega_0(p)\,,\,U\wedge\ol V\,\rangle$, $U, V\in T^{1,0}_pX$.
\end{definition}

Denote by $T^{*1,0}X$ and $T^{*0,1}X$ the dual bundles of
$T^{1,0}X$ and $T^{0,1}X$ respectively. Define the vector bundle of $(0,q)$ forms by
$T^{*0,q}X=\Lambda^q(T^{*0,1}X)$.
Let $D\subset X$ be an open set. Let $\Omega^{0,q}(D)$
denote the space of smooth sections of $T^{*0,q}X$ over $D$ and let $\Omega_0^{0,q}(D)$
be the subspace of $\Omega^{0,q}(D)$ whose elements have compact support in $D$. Similarly, if $E$ is a vector bundle over $D$, then we let $\Omega^{0,q}(D, E)$ denote the space of smooth sections of $T^{*0,q}X\otimes E$ over $D$ and let $\Omega_0^{0,q}(D, E)$ be the subspace of $\Omega^{0,q}(D, E)$ whose elements have compact support in $D$.

As in $S^1$-action case (see Section 2.3 in~\cite{HLM16}), for $u\in\Omega^{0,q}(X)$, we define
\begin{equation}\label{e-gue150508faIIq}
Tu:=\frac{\pr}{\pr\eta}\bigr(\eta^*u\bigr)|_{\eta=0}\in\Omega^{0,q}(X),
\end{equation}
where $\eta^*: T^{*0,q}_{\eta\circ x}X\To T^{*0,q}_{x}X$ is the pull-back map of $\eta$. Let $\ddbar_b:\Omega^{0,q}(X)\rightarrow\Omega^{0,q+1}(X)$ be the tangential Cauchy-Riemann operator.
Since the $\Real$-action is CR, as in $S^1$-action case (see Section 2.4 in~\cite{HLM16}), we have
\[T\ddbar_b=\ddbar_bT\ \ \mbox{on $\Omega^{0,q}(X)$}.\]

\begin{definition}\label{d-gue50508d}
Let $D\subset U$ be an open set. We say that a function $u\in C^\infty(D)$ is rigid if $Tu=0$. We say that a function $u\in C^\infty(X)$ is Cauchy-Riemann (CR for short)
if $\ddbar_bu=0$. We say that $u\in C^\infty(X)$ is rigid CR if  $\ddbar_bu=0$ and $Tu=0$.
\end{definition}

\subsection{Rigid CR bundles}
Let \((X,T^{1,0}X)\), \(\dim X=2n+d\), be a CR manifold of codimension \(d\in\N\) and CR dimension \(n\in\N\). The following definitions for CR vector bundles can be found in \cite{HN00}.
\begin{definition}\label{Def:CRVB}
	A complex vector bundle \((E,\pi,X)\) over \(X\) is called CR vector bundle if
	\begin{itemize}
		\item [(i)] \(E\) is a CR manifold of codimension \(d\),
		\item [(ii)] \(\pi\colon E\to X\) is a CR submersion,
		\item [(iii)] \(E\oplus E\ni(\xi_1,\xi_2)\to \xi_1+\xi_2\in E\) and \(\C\times E\ni(\lambda,\xi)\to \lambda \xi\in E\) are CR maps.
	\end{itemize}
	A smooth section \(s\in C^\infty(U,E)\) defined on an open set \(U\subset X\) is called CR section if the map \(s\colon U\to E\) is CR.
\end{definition}

Let \((E_1,\pi_1,X)\) and \((E_2,\pi_2,X)\) be two CR vector bundles over X. A map \(F\colon E_1\to E_2\) is called a CR bundle isomorphism if \(F\) is a  \(C^\infty\)-diffeomorphism such that \(F,F^{-1}\) are CR maps, \(\pi_2\circ F=\pi_1\) and \(F\) is fiberwise linear.


Given a CR vector bundle \((E,\pi,X)\) we find (see \cite{HN00}) the linear partial differential  operator \(\overline{\partial}^E_b\colon C^\infty(X,E)\to C^\infty(X,E\otimes T^{\ast 0, 1}X)\) satisfying
\begin{itemize}
	\item [(a)] \(\overline{\partial}^E_b(f\cdot s)=s\overline{\partial}_b(f)+f\overline{\partial}^E_b(s)\) for all \(f\in C^\infty(X)\) and \(s\in C^\infty(X,E)\),
	\item [(b)] \(s\in C^\infty(U,E)\) is a CR section if and only if \(\overline{\partial}^E_bs=0\).
\end{itemize}
\begin{definition}\label{Def:LocTriv}
 A CR vector bundle \((E,\pi,X)\) of rank \(r\) is called locally CR trivializable if for any point \(p\in X\)  there exists an open neighborhood \(U\subset X\) such that \(E|_U\) is CR vector bundle isomorphic to the trivial CR vector bundle \(U\times \C^r\).
\end{definition}

The following lemma is well-known.

\begin{lemma}\label{Lem:LocCR}
	Let \((E,\pi,X)\) be a CR vector bundle. The following are equivalent:
	\item [(i)] \((E,\pi,X)\) is locally CR trivializable,
	\item [(ii)] For any \(p\in X\) there exists a smooth frame \(\{f_1,\ldots,f_r\}\) of \(E|_U\) on an open neighborhood \(U\subset X\) around \(p\) such that \(f_1,\ldots,f_r\colon U\to E\) are CR sections.
\end{lemma}
	\begin{proof}
	Let \(p\in X\) be a point. Assuming that \((E,\pi,X)\) is locally CR trivializable we find an open neighborhood \(U\subset X\) around \(p\) and a CR bundle isomorphism \(F\colon  U\times\C^r\to E|_U\). For  \(1\leq j\leq r\) let \(e_j\in\C^r\) be the vector which has a one at the \(j\)-th position and zeros everywhere else. Then we have that \(x\mapsto (x,e_j)\) defines a CR map between \(U\) and \(U\times \C^r\). Putting \(f_j\colon U\to E|_U\), \(f_j(x)=F(x,e_j)\) it follows that \(f_j\) is a smooth CR map and since \(F\) is a bundle map we find that \(f_j\) is a CR section for any \(1\leq j\leq r\). For \(x\in U\) assume \(\sum_{j=1}^r\lambda_jf_j(x)=0\) for some \(\lambda_1,\ldots,\lambda_r\in\C\). We find \(0=F(x,(\lambda_1,\ldots,\lambda_r))\) and since \(F\) is a bundle isomorphism we must have \(\lambda_1=\ldots=\lambda_r=0\). Hence \(\{f_1(x),\ldots,f_r(x)\}\) is linear independent for any \(x\in U\).
	
	Now let \(\{f_1,\ldots,f_r\}\) be a smooth frame of \(E|_U\) such that \(f_j\colon U\to E|_U\) is a CR map for any \(1\leq j\leq r\). From (iii) in Definition~\ref{Def:CRVB} it follows that \(F\colon U\times \C^r\to E|_U\), \(F(x,(\lambda_1,\ldots,\lambda_r))=\sum_{j=1}^r\lambda_jf_j(x)\) is a CR map. By construction we have that \(F\) is a bundle isomorphism and since \(\{f_1,\ldots,f_r\}\) is a smooth frame we have that \(F\) is a diffeomorphism. Then we just need to show that \(dF(T^{1,0}(U\times\C^r))=T^{1,0}E|_U\) in order to prove that \(F^{-1}\) is a CR map. The map \(F\) is CR which means \(dF(T^{1,0}(U\times\C^r))\subset T^{1,0}E|_U\). Furthermore, we have that \(dF\) is injective at any point which implies \(\dim_\C dF(T^{1,0}(U\times\C^r))=n+r=\dim_\C T^{1,0}E|_U\) and the claim follows.
\end{proof}

\begin{remark}\label{Rmk:TransCR}
Let \(\{f_1,\ldots,f_r\}\) be a frame of \(E|_U\) for some open set \(U\subset X\). Then \(\{f_1,\ldots,f_r\}\) is called CR frame if any \(f_k\), \(1\leq k\leq r\), is a CR section. Given two CR frames of \(E|_U\) we find by (a) and (b) that the corresponding transition matrix is CR in the sense that any entry is a CR function.
\end{remark}

\begin{definition}\label{Def:CRBL}
Let $(X, T^{1, 0}X)$ be a CR manifold of codimension \(d\) and let \(T\in C^\infty(X,TX)\) be a CR vector field (that is \([T, C^\infty(X, T^{1,0}X)]\subset C^\infty(X, T^{1,0}X) \)).
 A CR bundle lift of \(T\) to \((E,\pi,X)\) is a linear partial differential operator \(T^E\colon   C^\infty(X,E)\to   C^\infty(X,E)\) (with smooth coefficients) such that
	\begin{itemize}
		\item [(i)] \(T^{E}(f\cdot s)=T(f)\cdot s+fT^E(s)\) for all \(f\in C^\infty(X)\) and \(s\in  C^\infty(X,E)\),
		\item [(ii)] \([T^{E},\overline{\partial}^E_b]=0\).
	\end{itemize}
\end{definition}
In order to define \([T^{E},\overline{\partial}^E_b]\) we need to define \(T^{E}\) on \((0,1)\) forms with values in \(E\) first. But this definition follows immediately from the fact that any $w\in  C^\infty(X, E\otimes T^{\ast 0, 1}X)$ locally can be written $w=\sum_{j=1}^r f_j \otimes\omega^j$ where $\{\omega^j\}$ are $(0, 1)$-forms and $\{f_j\}$ are local frames of $E$  and that \(T\) is defined also for \((0 ,q)\)-forms using the Lie derivative.
\begin{definition}\label{Def:RigidCVB}
	Let $(X, T^{1, 0}X)$ be a CR manifold of codimension \(d\) and let \(T\in  C^\infty(X,TX)\) be a CR vector field.
	A CR vector bundle \((E,\pi,X)\) of rank \(r\) over \(X\) with a CR bundle lift \(T^E\) of \(T\)  is called rigid CR (with respect to \(T^E\)) if for every point \(p\in X\) there exists an open neighborhood \(U\) around \(p\) and a CR frame \(\{f_1,\ldots,f_r\}\) of \(E|_U\) with \(T^E(f_j)=0\) for \(1\leq j\leq r\).
\end{definition}
A section $s\in  C^\infty(X, E)$ is called a rigid CR section if $T^E s=0$ and $\overline\partial^E_b s=0$. The frame $\{f_j\}_{j=1}^r$ in Definition~\ref{Def:RigidCVB}  is called a rigid CR frame of $E|_U$.
Note that it follows from Lemma~\ref{Lem:LocCR} that any rigid CR vector bundle is locally CR trivializable.

\begin{lemma}\label{Thm:LocTRigid}
	Let \((E,\pi,X)\) be CR vector bundle over a CR manifold \((X,T^{1,0}X)\) of codimension \(d\) and let  \(T\in  C^\infty(X,TX)\) be a CR vector field. The following are equivalent:
	\begin{itemize}
		\item [(i)] \(T\) has a CR bundle lift \(T^E\) such that \((E,\pi,X)\) is rigid CR with respect to \(T^E\).
		\item [(ii)] There exist an open cover \(\{U_j\}_{j\in\N}\) of \(X\) and CR frames \(\{f^j_1,\ldots,f_r^j\}\) for \(E|_{U_j}\), \(j\in\N\), such that the corresponding transition matrices are rigid CR in the sense that any entry is a rigid CR function.
	\end{itemize}
\end{lemma}
Recall that a function \(f\in C^\infty(X)\) is rigid if \(Tf=0\) holds.

\begin{proof}
	In order to prove "(ii)\(\Rightarrow\) (i)`` define a CR bundle lift \(T^E\) of \(T\) as follows: Given a smooth section \(s\in  C^\infty(X,E)\) and a point \(p\in X\) write  \(s|_{U_j}=\sum_{k=1}^ra^j_kf_k^j\) for any \(j\in \N\) with \(p\in U_j\) where \(a_k^j\) are smooth functions  on \(U_j\). Then define \(T^E(s)(p)=\sum_{k=1}^rT(a^j_k)f_k^j\). The definition is independent of \(j\) since the transition matrices are rigid. Since \(T\) satisfies the Leibniz rule the same holds for \(T^E\) and since \([T,\overline{\partial}_b]=0\) and the local frames
	\(\{f^j_1,\ldots,f_r^j\}\), \(j\in\N\),  are CR we find \([T^{E},\overline{\partial}^E_b]=0\). By construction we find that the frames \(\{f^j_1,\ldots,f_r^j\}\) are rigid CR and hence that  \((E,\pi,X)\) is rigid CR with respect to \(T^E\).
	
	The implication ''(i)\(\Rightarrow\) (ii)`` follows from Definition~\ref{Def:RigidCVB}:
	For any point \(p\in X\) we find an open neighborhood \(U_p\) around \(p\) and a CR frame \(\{f_1^p,\ldots,f_r^p\}\) of \(E|_{U_p}\) with \(T^E(f_l^p)=0\) for \(1\leq l\leq r\). Since \(X\) is a manifold we can choose \(\{p_j\}_{j\in \N}\) such that \(\{U_{p_j}\}_{j\in\N}\) is an open cover of \(X\). For \(j\in\N\) and \(1\leq l\leq r\) put \(f_l^j:=f_l^{p_j}\) and \(U_j:=U_{p_j}\). Given \(j,k\in\N\) with \(U_k\cap U_j\neq\emptyset\)  let \(A\) denote the transition matrix between the frames \(\{f_1^{j},\ldots,f_r^{j}\}\) and \(\{f_1^{k},\ldots,f_r^{k}\}\) that is
	\[(f_1^{j},\ldots,f_r^{j})=(f_1^{k},\ldots,f_r^{k})A.\]
	It follows from Remark~\ref{Rmk:TransCR} that \(A\) is CR. Furthermore, we find
	\[0=(T^E(f_1^{j}),\ldots,T^E(f_r^{j}))=(T^E(f_1^{k}),\ldots,T^E(f_r^{k}))A+(f_1^{k},\ldots,f_r^{k})TA=(f_1^{k},\ldots,f_r^{k})TA.\]
	Since \(\{f_1^{k},\ldots,f_r^{k}\}\) is a frame we have \(TA=0\), that is the transition matrix is rigid and CR.
\end{proof}

Let $(X, T^{1, 0}X)$, \(\dim X=2n-1\), be a CR manifold of codimension one and CR dimension \(n-1\) with a transversal CR $\mathbb R$-action. Let $T$ be the infinitesimal generator of the $\mathbb R$-action. In this paper we will make systematic use of appropriate coordinates introduced by
Baouendi-Rothschild-Treves~\cite[Theorem II.1, Proposition I.2]{BRT85}. For each point $p\in X$
there exist a coordinate neighborhood $U$ with coordinates $(x_1,\ldots,x_{2n-1})$,
centered at $p=0$, and $\varepsilon>0$, $\varepsilon_0>0$,
such that, by setting $z_j=x_{2j-1}+ix_{2j}$, $j=1,\ldots,n-1$, $x_{2n-1}=\eta$
and $D=\{(z, \eta)\in U: \abs{z}<\varepsilon, |\theta|<\varepsilon_0\}\subset U$, we have
\begin{equation}\label{e-can1}
T=\frac{\partial}{\partial\eta}\:\:\text{on $D$},\\
\end{equation}
and the vector fields
\begin{equation}\label{e-can2}
Z_j=\frac{\partial}{\partial z_j}-i\frac{\partial\phi}{\partial z_j}(z)\frac{\partial}{\partial\eta},
\:\:j=1,\ldots,n-1,
\end{equation}
form a basis of $T_x^{1,0}X$ for each $x\in D$, where $\phi\in C^\infty(D,\mathbb R)$
is independent of $\eta$.
We call $(x_1,\ldots,x_{2n-1})$ canonical coordinates, $D$ canonical coordinate patch and
$(D,(z,\eta),\phi)$ a BRT trivialization. The frames \eqref{e-can2} are called BRT frames. We can also define BRT frames on the bundle $T^{*0,q}X$. We sometime write $(D,x=(x_1,\ldots,x_{2n-1}))$ to denote canonical coordinates.

\begin{example}
Let $X$ be a compact CR manifold  with a  transversal CR
$\mathbb R$-action. Let $T$ be the infinitesimal generator of the $\mathbb R$-action. We study here the bundle $T^{1, 0}X$
by using the canonical BRT coordinates \cite[Theorem II.1, Proposition I.2]{BRT85}. In particular, we will show that the BRT coordinates give rise to a CR structure on $T^{1, 0}X$ and a CR bundle lift of \(T\), such that \(T^{1,0}X\) becomes a rigid CR vector bundle.
Let $(D, (z,\theta),\phi)$ be a BRT trivialization defined in
\eqref{e-can2}. Then on $D$,
\begin{equation}\label{e-can11}
\begin{split}
T&=\frac{\partial}{\partial\theta},\\
Z_j&=\frac{\partial}{\partial z_j}-
i\frac{\partial\phi}{\partial z_j}(z,\overline{z})\frac{\partial}{\partial\theta},
\:\:j=1,\ldots,n-1,
\end{split}
\end{equation}
where $\{Z_j: j=1, \ldots, n-1\}$ is a frame of $T^{1, 0}X$ over $D$.   Let $(\tilde D, (w, \eta), \tilde \phi)$ be another BRT trivialization. Then on $\tilde D$,
\begin{equation}\label{e-can22}
\begin{split}
T&=\frac{\partial}{\partial\eta},\\
\tilde Z_j&=\frac{\partial}{\partial w_j}-
i\frac{\partial\tilde\phi}{\partial w_j}(w,\overline{w})\frac{\partial}{\partial\eta},
\:\:j=1,\ldots, n-1,
\end{split}
\end{equation}
where $\{\tilde Z_j: j=1, \ldots, n-1\}$ is a frame of $T^{1, 0}X$ over
$\tilde D$.
We have on $D\cap\tilde{D}$,
\begin{equation}\label{I.32}
\tilde{Z}_j=\sum_{k=1}^{n-1}c_{j,k}Z_k
\end{equation}
where $c_{j,k}\in C^\infty(D\cap \tilde{D})$ are rigid CR functions. Write $E=T^{1, 0}X$. The local frames give rise to a CR vector bundle structure on \(E\). Then $T$ will admit a natural CR bundle lift $T^E$ on $E$. In fact, for any $f\in  C^\infty(X, E)$ we can write $f=\sum_{j=1}^{n-1} f_j Z_j$ and one can define
\begin{equation}\label{example-2019-05-13a}
T^E f=\sum_{j=1}^{n-1} (Tf_j )Z_j.
\end{equation}
Moreover, since $[T, \overline\partial_b]=0$ then it follows from (\ref{example-2019-05-13a}) that $[T^E, \overline\partial^E_b]=0.$
\end{example}

The goal of our paper is to prove a Kodaira embedding theorem, so to work with very ample line bundles, whose global CR sections give an embedding in the projective space. Such bundles are locally CR trivializable, so we restrict here to CR vector bundles which are locally CR trivializable.
The following lemma can be seen as a variant of Proposition~2.7 in \cite{HLM16} for bundle lifts of the vector field \(T\).

\begin{lemma}\label{Lem:RigidFrame}
	Let $(X, T^{1, 0}X)$ be a CR manifold of codimension one with a transversal CR $\mathbb R$-action. Let $T$ be the infinitesimal generator of the $\mathbb R$-action.
	Let \((E,\pi,X)\) be a locally CR trivializable CR vector bundle of rank \(r=1\). Assume that \(T^E\) is a CR bundle lift of \(T\) to \((E,\pi,X)\). Then \((E,\pi,X)\) is rigid CR. More precisely,  for any \(p\in X\) there exist an open neighborhood \(U\subset X\) around \(p\) and a CR frame \(\{f\}\) of \(E|_U\) with \(T^E(f)=0\).
\end{lemma}

\begin{proof}
	Using Lemma~\ref{Lem:LocCR} we find an open neighborhood \(V\subset X\) around \(p\) and a CR frame \(\{s\}\) of \(E|_V\). Any other smooth frame \(\{f\}\) on \(V\) can be written as \(f=sA\) where \(A\colon V\to \C\setminus\{0\}\) is smooth. Furthermore, we can write \(T^E(s)=sB\) with \(B\colon V\to \C\) smooth. Since \(T\) is non vanishing, we can solve the linear partial differential equation \(T(A)=-BA\) in \(A\) with \(A(p)=1\)  in a small neighborhood \(V'\) of \(p\) with \(A(x)\in \C\setminus\{0\}\) for any \(x\in V'\). Then \(\{f\}\) defined by \(f=sA\) is a frame of \(E|_{V'}\) with
	\(T^E(f)=s(TA+AB)=0\).
	It remains to show that we can find a solution \(A\) such that \(\{f\}\) is a CR frame, that is \(\overline{\partial}_bA=0\). With \([T^E,\overline{\partial}^E_b]=0\) we find \(\overline{\partial}_bB=0\) and since \([T,\overline{\partial}_b]=0\) we have \(T(\overline{\partial}_bA)=B(\overline{\partial}_bA)\). Therefore we have to find a hypersurface \(H\) around \(p\) transversal with respect to \(T\) and initial Data \(A_0\) on \(H\) such that the solution of the transport equation \(T(A)=-BA\) with \(A=A_0\) on \(H\) satisfies \(\overline{\partial}_bA=0\) on \(H\). Then it follows from \(T(\overline{\partial}_bA)=-B(\overline{\partial}_bA)\) that \(\overline{\partial}_bA=0\) holds in an open neighborhood around \(p\). Choose BRT coordinates \(((z,t)\in P\times I,\varphi)\) on an open neighborhood \(U'\) around \(p\) such that \(P\) is an open polydisc in some \(\C^{n-1}\), \(I\subset\R\) an open interval around \(0\) and identify \(U'\) with \(P\times I \) where \(p\) corresponds to \((0,0)\in P\times I\). Set \(H=P\times\{0\}\) and write \(A=A(z,t)\), \(B=B(z,t)\).  If \(A\) is a solution of \(T(A)=-BA\) with \(\overline{\partial}_bA=0\) we must have \(\overline{\partial} A(z,0)=-i(\overline{\partial}\varphi) B(z,0)A(z,0)\) on \(P\) where \(\overline{\partial}=\sum_{j=1}^{n-1}d\overline{z}_j\wedge\frac{\partial}{\partial \overline{z}_j}\). From \(\overline{\partial}_bB=0\) we find \(\overline{\partial}B=i(\overline{\partial}\varphi) TB\) and hence \(\overline{\partial}((\overline{\partial}\varphi(z)) B(z,0))=0\) on \(P\). So let \(g\in C^\infty(P)\) be a smooth solution of \(\overline{\partial} g(z)=-i(\overline{\partial}\varphi) B(z,0)\) with \(g(0)=0\) and set \(A_0(z,t)=\exp(g(z))\). Let \(A\) be the solution of \(T(A)=-BA\) with \(A=A_0\) on \(H\).  By construction we have \(\overline{\partial}_bA=0\) on \(H\) and since \(T(\overline{\partial}_bA)=B(\overline{\partial}_bA)\) we have \(\overline{\partial}_bA=0\) in a neighborhood of \(p\). Since \(A(p)=A(0,0)=\exp(g(0))=1\) we find that \(A\) is non vanishing in a neighborhood \(U\) around \(p\). Then \(f:=sA\) is the desired frame for \(E|_U\).
\end{proof}

\begin{definition}\label{d-gue150514f}
Let $E$ be a rigid vector bundle over $X$. Let $\langle\,\cdot\,|\,\cdot\,\rangle_E$ be a Hermitian metric on $E$. We say that $\langle\,\cdot\,|\,\cdot\,\rangle_E$ is a rigid Hermitian metric if for every local rigid frame $f_1,\ldots, f_r$ of $E$, we have $T\langle\,f_j\,|\,f_k\,\rangle_E=0$, for every $j,k=1,2,\ldots,r$.
\end{definition}


In order to simplify the notation we will denote by $\overline\partial_b, T$ the operators $\overline\partial_b^E, T^E$ where $E$ is any rigid CR vector bundle on $X$.
Let $(X, T^{1, 0}X)$ be a CR manifold of codimension one with a transversal CR $\mathbb R$-action and let $T$ be the infinitesimal generator of the $\mathbb R$-action.
Consider a locally CR trivializable CR line bundle $L$  over $X$ with a CR bundle lift  of \(T\). By Lemma \ref{Lem:RigidFrame} we find that \(L\) is rigid CR with respect to that bundle lift. Hence there exists
an open covering $(U_j)^N_{j=1}$ and a family of rigid CR trivializing frames $\{s_j\}_{j=1}^N$ with each $s_j$ defined on $U_j$ and the transition functions between different rigid CR frames are rigid CR functions.
Let $L^k$ be the $k$-th tensor power of $L$.
Then $\{s_j^{k}\}^N_{j=1}$ is a family of  rigid CR trivializing frames on each $U_j$.
Let
$\overline\partial_b^{L^k}:\Omega^{0,q}(X, L^k)\rightarrow\Omega^{0,q+1}(X, L^k)$ be the tangential Cauchy-Riemann operator.
Since $L^k$ is rigid  CR we have $\overline\partial_b f=\overline\partial_b f_j\otimes s_j^k$, $Tf=(Tf_j)\otimes s_j^k$ for any $f=f_j\otimes s_j^k\in\Omega^{0, q}(X, L^k)$ and
\begin{equation}\label{e-gue150508d}
T\ddbar_b=\ddbar_bT\ \ \mbox{on $\Omega^{0,q}(X,L^k)$}.
\end{equation}

Let $h^L$ be a Hermitian fiber metric on $L$. The local weight of $h^L$
with respect to a local rigid CR trivializing section $s$ of $L^L$ over an open subset $D\subset X$
is the function $\Phi\in C^\infty(D, \mathbb R)$ for which
\begin{equation}\label{e-gue150808g}
|s(x)|^2_{h^L}=e^{-2\Phi(x)}, x\in D.
\end{equation}
We denote by $\Phi_j$ the weight of $h^L$ with respect to $s_j$.

\begin{definition}\label{d-gue150808g}
Let $L$ be a rigid CR line bundle and let $h^L$ be a Hermitian metric on $L$.
The curvature of $(L,h^L)$ is the the Hermitian quadratic form $R^L=R^{(L,h^L)}$ on $T^{1,0}X$
defined by
\begin{equation}\label{e-gue150808w}
R_p^L(U, V)=\,\big\langle d(\overline\partial_b\Phi_j-\partial_b\Phi_j)(p),
U\wedge\overline V\,\big\rangle,\:\: U, V\in T_p^{1,0}X,\:\: p\in U_j.
\end{equation}
\end{definition}

Due to \cite[Proposition 4.2]{HM09}, $R^L$ is a well-defined global Hermitian form,
since the transition functions between different frames $s_j$ are annihilated by $T$.

\begin{definition}\label{d-gue150808gI}
We say that $(L,h^L)$ is positive if there is an interval $I\subset\Real$ such that the associated curvature $R^L_x-2s\mathcal{L}_x$ is positive definite
at every $x\in X$, for every $s\in I$.
\end{definition}


\section{The relation between $\Real$-action and torus action on CR manifolds} \label{s-gue170810}

In this section we state and proof our main results about \(\Real\)-actions on a CR manifold \(X\) (see Theorem \ref{thm:mainthm}). It turns out that if the \(\Real\)-action is CR transversal and  \(X\) is either strongly pseudoconvex or admits a rigid positive CR line bundle there are only two cases which need to be considered. In particular, we find out that  the \(\Real\)-action does not come from a CR torus action if there exists an orbit which is a closed but non-compact subset of \(X\) and in that case all the orbits have this property (see Corollary \ref{cor:psc} and \ref{c-gue170811}). If \(X\) is in addition compact it is easy to see that the \(\Real\)-action is always induced by  a CR torus action.
\subsection{Some facts in Riemannian geometry}\label{s-gue170810I}

Let $(X, g)$ be a connected Riemannian manifold with metric $g$ and denote by  \(\operatorname{Iso}(X,g)\) the group of isometries from \((X,g)\) onto itself, that is \(F\in \operatorname{Iso}(X,g)\) if and only if \(F\) is a \(C^\infty\)-Diffeomorphism and \(F^*g=g\).  A Lie group is always assumed to be finite dimensional. The following result is well-known (see~\cite{MS39}).

\begin{theorem}\label{thm:steenrod}
	We have that \(\operatorname{Iso}(X,g)\) is a Lie transformation group acting on \(X\). More precisely, \(\operatorname{Iso}(X,g)\) together with the composition of maps carries the structure of a Lie group such that the map
	\begin{align*}
	\operatorname{Iso}(X,g)\times X \ni (F,x)\mapsto F(x)\in X
	\end{align*}
	is of class \(C^1\).\\
	Furthermore, assuming that \(X\) is compact it follows that  \(\operatorname{Iso}(X,g)\) is compact too.
\end{theorem}

\begin{lemma}\label{lem:steenrod}
	In the situation of Theorem \ref{thm:steenrod}  we have that for every \(v\in \C TX\) the map \(Q_v\colon\operatorname{Iso}(X,g)\rightarrow \C TX \), \(Q_v(F)=dF_{\pi(v)}v\) is continuous. Here \(\pi:\C TX\rightarrow X\) denotes the standard projection and all fibrewise linear maps on \(TX\) are extended \(\C\)-linearly to \(\C TX\).
\end{lemma}

\begin{proof}
	The proof follows immediately from Lemma 7 in~\cite{MS39}.
\end{proof}

\begin{lemma}\label{lem:propermap}
	The map \(\operatorname{Iso}(X,g)\times X \ni (F,x)\mapsto (x,F(x))\in X\times X\) is proper.
\end{lemma}
\begin{proof}
	see Satz 2.22 in~\cite{Sch08}.
\end{proof}

\subsection{Application to CR geometry}\label{s-gue170810II}

Let \((X,T^{1,0}X)\) be a connected CR manifold and denote by \(\operatorname{Iso}(X,g)\) the group of isometries on \(X\) with respect to some Riemannian metric \(g\).  Let \(\operatorname{Aut}_{\operatorname{CR}}(X)\) be the group of CR automorphisms on \(X\), that is  \(F\in\operatorname{Aut}_{\operatorname{CR}}(X)\) if and only if \(F\colon X\rightarrow X\) is a \(C^{\infty}\)-Diffeomorphism satisfying \(dF(T^{1,0}X)\subset T^{1,0}X\).
\begin{lemma}
 We have that \(\operatorname{Iso}(X,g)\cap \operatorname{Aut}_{\operatorname{CR}}(X)\) is a Lie group. Furthermore, assuming that \(X\) is compact implies that \(\operatorname{Iso}(X,g)\cap \operatorname{Aut}_{\operatorname{CR}}(X)\) is a compact Lie group.
\end{lemma}
\begin{proof}
	Obviously, \(\operatorname{Iso}(X,g)\cap \operatorname{Aut}_{\operatorname{CR}}(X)\) is a subgroup of \(\operatorname{Iso}(X,g)\). We only need to show that \(\operatorname{Iso}(X,g)\cap \operatorname{Aut}_{\operatorname{CR}}(X)\) is a topologically closed subset of \(\operatorname{Iso}(X,g)\). Then, by Theorem~\ref{thm:steenrod}, Cartan's closed subgroup theorem and the fact that a closed subset of a compact set is again compact, the result follows. Recall that Cartan's closed subgroup theorem states that if $H$ is a closed subgroup of a Lie group $G$, then $H$ is an embedded Lie group with the relative topology being the same as the group topology.\\
	We have that \(T^{1,0}X\) is a closed subset of \(\C TX\). Then by Lemma \ref{lem:steenrod} we have that for every \(v\in T^{1,0}X\) the set \(Q_v^{-1}(T^{1,0}X)\) is a closed subset of \(\operatorname{Iso}(X,g)\) and hence \(H:=\bigcap_{v\in T^{1,0}X}Q_v^{-1}(T^{1,0}X)\) is a closed subset of \(\operatorname{Iso}(X,g)\). Moreover, by definition a \(C^\infty\)-Diffeomorphism \(F\colon X\rightarrow X\) is CR if and only if  \(dF_{\pi(v)}v\in T^{1,0}X\) holds for all \(v\in T^{1,0}X\). Hence, \(H=\operatorname{Iso}(X,g)\cap \operatorname{Aut}_{\operatorname{CR}}(X)\) which proofs the claim.
\end{proof}

Now assume that \((X,T^{1,0}X)\) is equipped with a CR \(\R\)-action, i.e.~a Lie group homomorphism \(\gamma:\R\rightarrow \operatorname{Aut}_{\operatorname{CR}}(X)\).
\begin{theorem}\label{thm:mainthm}
	Let  \((X,T^{1,0}X)\) be a connected CR manifold equipped with a CR \(\R\)-action. Assume that there exists a Riemannian metric \(g\) on \(X\), such that the \(\R\)-action acts by isometries with respect to this metric. Then exactly one of the following two cases will appear:
	\begin{itemize}
		\item[case 1:] All orbits are closed subsets and non compact.
		\item[case 2:] \(\overline{\gamma(\R)}\) is a torus in \(\operatorname{Iso}(X,g)\cap \operatorname{Aut}_{\operatorname{CR}}(X)\). In other words, the \(\R\)-action comes from a CR torus action.
	\end{itemize}
	Here \(\overline{\gamma(\R)}\) is the closure of \(\gamma(\R)\) taken in \(\operatorname{Iso}(X,g)\cap \operatorname{Aut}_{\operatorname{CR}}(X)\).
\end{theorem}

\begin{remark}
	Note that we neither assume that the \(\R\)-action is transversal or locally free nor that the manifold is compact. However, if we additionally assume that \(X\) is compact, we find that the first case cannot appear and hence the \(\R\)-action is induced by a CR torus action.
\end{remark}

\begin{proof}[Proof of Theorem \ref{thm:mainthm}.]
	If \(\gamma\) fails to be injective, we find that the \(\R\)-action is either constant or reduces to an \(S^1\)-action. In both cases there is nothing to show. So let us assume that \(\gamma\) is injective.

We have that \(\operatorname{Iso}(X,g)\cap \operatorname{Aut}_{\operatorname{CR}}(X)\) is a Lie group and that  \(\overline{\gamma(\R)}\) is a topologically closed, abelian subgroup. Hence, \(\overline{\gamma(\R)}\) is an abelian Lie group and thus can be identified with \(V\times\mathcal{T}\), where \(V\) is a finite dimensional real vector space and \(\mathcal{T}\) is some torus.

In the case \(\overline{\gamma(\R)}=\gamma(\R)\) we find by dimensional reasons and because \(\gamma\) is injective that \(\mathcal{T}=\{\operatorname{id}\}\) and \( V\simeq \R\) holds. Take a point \(p\in X\) and consider the map \(\tilde{\gamma}\colon\R\rightarrow X\), \(t\mapsto \gamma(t)(p)\). Since \(\gamma(\R)\) is a closed subset of \(\operatorname{Iso}(X,g)\cap \operatorname{Aut}_{\operatorname{CR}}(X)\) which is closed in \(\operatorname{Iso}(X,g)\) and by Lemma \ref{lem:propermap} we have that \(\tilde{\gamma}\) is proper. Furthermore, \(\tilde{\gamma}\) is injective, because otherwise it would be periodic or constant what contradicts the properness. Summing up we have that \(\tilde{\gamma}\) is continuous, injective and proper and hence it is an embedding. Since \(p\in X\) was chosen arbitrary, case 1 follows.\\
	Given the case \(\overline{\gamma(\R)}\neq\gamma(\R)\) we will show that \(V=\{0\}\) holds. Denote the action \(\mathcal{T}\curvearrowright X\) by $(e^{i\theta_1},\ldots,e^{i\theta_d})$. It is not difficult to see that the $\Real$-action $\gamma$ is given by
\[t\mapsto (tv, e^{i\alpha_1t},\ldots,e^{i\alpha_dt}),\ \ \mbox{for some $v\in V$ and $(\alpha_1,\ldots,\alpha_d)\in\Real^d$}\]
and hence the projection of \(\gamma\) onto \(V\) is of the form \(t\mapsto tv \), for some \(v\in V\). First consider the case \(v\neq 0\). For \(h=(w,\lambda)\in \overline{\gamma(\R)}=V\times \mathcal{T}\) choose an open neighbourhood \(U\) in \(V\times\mathcal{T}\) around \(h\) with compact closure. We find that there exists \(t_0>0\) such that \(\gamma(t) \notin \overline{U}\) for \(|t|>t_0\). So we have \(h\in\overline{\gamma([-t_0,t_0])}\). Since \(\gamma([-t_0,t_0])\) is compact we conclude \(h\in\gamma([-t_0,t_0])\). Thus,  \(\overline{\gamma(\R)}\neq\gamma(\R)\) implies \(v=0\) which leads to \(V=\{0\}\) since  \(\gamma(\R)\) has to be dense in \(V\times\mathcal{T}\).
\end{proof}

\begin{example}
	Consider \(X=\C_z\times\R_s\), \(T^{1,0}X=\C\cdot (\frac{\partial}{\partial z}+i\frac{\partial \varphi}{\partial z}(z)\frac{\partial}{\partial s})\) for some function \(\varphi\in C^\infty(\C)\). Then \((X,T^{1,0}X)\) is a CR manifold and \((t,(z,s))\mapsto (z,s+t)\) defines a transversal CR \(\R\)-action. We observe that this action is not induced by a CR torus action.
\end{example}

\begin{corollary}\label{cor:psc}
	Let \((X,T^{1,0}X)\) be a connected strongly pseudoconvex CR manifold equipped with a transversal CR \(\R\)-action. We have that the \(\R\)-action comes from a CR torus action if and only if at least one of the following conditions is satisfied:
	\begin{itemize}
		\item[a)] there exists an orbit which is a non-closed subset of \(X\),
		\item[b)] there exists an orbit which is compact.
	\end{itemize}
\end{corollary}
\begin{proof}
	Using Theorem \ref{thm:mainthm} we just need to construct an \(\R\)-invariant Riemannian metric on \(X\). Let \(T\) denote the vector field induced by the \(\R\)-action. Since the action is transversal we have
	\[\C TX=\C T\oplus T^{1,0}X\oplus T^{0,1} X.\]
	Let \(P,P^{1,0},P^{0,1}\) be the projections which belong to the decomposition above and denote by \(\omega_0\) the real one form which satisfies \(\omega_0(T)=-1\) and \(\omega_0(T^{1,0}X\oplus T^{0,1} X)=0\). Since \((X,T^{1,0}X)\) is strongly pseudoconvex we have that \(\frac{i}{2}d\omega_0\) induces a Hermitian metric on \(T^{1,0}X\). Now set
	\[g=\omega_0\otimes\omega_0+\frac{i}{2}d\omega_0\left(P^{1,0}(\cdot),P^{0,1}(\cdot)\right).\]
	Identifying \(TX\) with \(1\otimes TX \subset \C TX\) we find that \(g\) defines a Riemannian metric on \(X\).
	Using BRT trivializations (see Section \ref{s:prelim}) one can check that the Lie derivative of \(g\) with respect to \(T\) vanishes and hence that \(g\) is \(\R\)-invariant. Then the claim follows from Theorem \ref{thm:mainthm}.
\end{proof}

\begin{corollary}\label{c-gue170811}
	Let \((X,T^{1,0}X)\) be a connected CR manifold equipped with a transversal CR \(\R\)-action. Assume that \(L\rightarrow X\) is a rigid positive CR line bundle over \(X\). We have that the \(\R\)-action comes from a CR torus action if and only if at least one of the following conditions is satisfied:
	\begin{itemize}
		\item[a)] there exists an orbit which is a non-closed subset of \(X\),
		\item[b)] there exists an orbit which is compact.
	\end{itemize}
\end{corollary}

\begin{proof}
	We have that \(L\) is an \(\R\)-invariant  Hermitian CR line bundle, which implies that the fibrewise metric on \(L\) is \(\R\)-invariant. Therefore its curvature, which is a smooth \((1,1)\)-form \(R^L\) on \(X\), is \(\R\)-invariant. Using the positivity of $R^L-2s\mathcal{L}$, for some $s\in\mathbb R$, we can proceed similar to the proof of Corollary~\ref{cor:psc} replacing \(\frac{i}{2}d\omega_0\) by $R^L-2s\mathcal{L}$. Thus, the Corollary follows from Theorem~\ref{thm:mainthm}.
\end{proof}

\begin{remark}
	Note that we do not assume compactness of \(X\) in the corollaries above. When \(X\) is compact, at least one of the the conditions a) and b) is automatically satisfied. Note that under the additional assumption that \(X\) is compact the conclusion of Corollary~\ref{cor:psc} can be  found in \cite{Le92}.
\end{remark}

From Corollary~\ref{cor:psc} and the equivariant embedding theorem established in~\cite{HHL16}, we  can prove

\begin{theorem}\label{thm:equivariantembeddingpseudoconvex}
		Let  $X$ be a connected compact strongly pseudoconvex CR manifold equipped with a transversal CR $\Real$-action $\eta$, $\eta\in\Real$: $\eta: X\to X$, $x\mapsto \eta\circ x$. Then, there exists $N\in\mathbb N$, $\nu_1,\ldots,\nu_N\in\Real$ and a CR embedding $\map{\Phi=(\Phi_1,\ldots,\Phi_N)}{X}{\Complex^N}$
		such that \[\Phi(\eta\circ x)=(e^{i\nu_1\eta}\Phi_1(x),\ldots,e^{i\nu_N\eta}\Phi_N(x))\] holds for all $x\in X$ and $\eta\in\Real$. In other words, $\Phi$ is eqivariant with respect to the holomorphic $\Real$-action $\eta\circ z=(e^{i\nu_1\eta}z_1,\ldots,e^{i\nu_N\eta}z_N)$ on $\C^N$.
	\end{theorem}

\begin{proof}
		By the assumptions we can apply Corollary \ref{cor:psc}. Since \(X\) is compact we have that the \(\R\)-action is a subaction of a CR torus action $T^r\curvearrowright X$ denoted by \((e^{i\tau_1},\ldots,e^{i\tau_r})\). We may assume that its rank \(r\) satisfies \(r>1\), because otherwise we have that the \(\R\)-action reduces to an \(S^1\)-action.
		
		Consider the vector fields $T_1,\ldots,T_r$ on $X$ given by
		\[(T_j)_x=\frac{\partial}{\partial \tau_j}(1,\ldots,1,e^{i\tau_j},1,\ldots,1)\circ x\mid_{\tau_j=0}.\]
		Let \(T\) denote the vector field induced by the \(\R\)-action. We find \(T=\sum_{j=1}^r\lambda_jT_j\) for some real numbers \(\lambda_1,\ldots,\lambda_r\in\R\).
		 By assumption, \(T\) is transversal and hence we find \(\tilde{\lambda}_1,\ldots,\tilde{\lambda}_r\in\mathbb Q\) (\(\tilde{\lambda}_j\) close to \(\lambda_j\)), 
		such that \(\tilde{T}=\sum_{j=1}^r\tilde{\lambda_j}T_j\) is transversal. Since \(\tilde{\lambda}_j\in\mathbb Q\), \(j=1,\ldots,r\), we have that \(\tilde{T}\) defines a transversal CR \(S^1\)-action on \(X\) and after rescaling \(\tilde{T}\) we can achieve that the \(S^1\)-action can be represented by \((e^{i\theta}, x)\mapsto e^{i\theta}\circ x\) with \(X_{\text{reg}}:=\set{x\in X;\, e^{i\theta}\circ x\neq x, \forall\theta\in]0,2\pi[}\neq\emptyset\) (we denote the rescaled vector field by \(T_0\)). Choose a \(T^r\)-invariant Hermitian metric on \(X\) with
		\[T^{1,0}X\perp T^{0,1}X\text{, } T_0\perp T^{1,0}X\oplus T^{0,1}X \text{ and } \|T_0\|=1.\]
Denote the space of CR functions for eigenvalue \(m\in\mathbb N\) with respect to the \(S^1\)-action \(e^{i\theta}\) by \(\mathcal{H}^0_{b,m}(X)\), that is
		\begin{align*}
		\mathcal{H}^0_{b,m}(X)&=\{f\in C^\infty(X);\, \ddbar_bu=0,\ \ T_0f=imf\}\\
		&=\{f\in C^\infty(X);\, \ddbar_bf=0,\ \ f(e^{i\theta}\circ x)=e^{im\theta}f(x)\text{, }\forall x\in X,\theta\in\R\}.
		\end{align*}
		The \(S^1\)-action is transversal, so we have \(\dim\mathcal{H}^0_{b,m}(X)<\infty\).
		Since \([T_0,T_j]=0\) for \(j=1,\ldots,r\) (or in other words, the \(S^1\)-action  commutes with the \(T^r\)-action) we find a decomposition
		\begin{align}\label{Eq0}
		\mathcal{H}^0_{b,m}(X)=\bigoplus_{\alpha\in\mathbb Z^r}\mathcal{H}_{b,(m,\alpha)}(X)
		\end{align}
		where
		\begin{align*}
		\mathcal{H}^0_{b,(m,\alpha)}(X)&=\{f\in \mathcal{H}^0_{b,m}(X),\ \  T_j f=i\alpha_jf, 1\leq j\leq r \}\\
		&=\{f\in \mathcal{H}^0_{b,m}(X),\ \ f((e^{i\tau_1},\ldots,e^{i\tau_r})\circ x)=e^{i\alpha\cdot\tau}f(x)\text{, }\forall x\in X,\tau\in\R^r\}
		\end{align*}
		with \(\alpha\cdot\tau=\alpha_1\tau_1+\ldots+\alpha_r\tau_r\). Furthermore, the decomposition (\ref{Eq0}) is orthogonal with respect to the \(L^2\) inner product coming from the \(T^r\)-invariant Hermitian metric. Given \(f\in\mathcal{H}^0_{b,(m,\alpha)}\) we find \(f(\eta\circ x)=e^{i(\lambda_1\alpha_1+\ldots+\lambda_r\alpha_r)\eta}f(x)\) for all \(x\in X\) and \(\eta\in\R\) because \(T=\sum_{j=1}^r\lambda_jT_j\) and hence \(\eta\circ x=(e^{i\lambda_1\eta},\ldots,e^{i\lambda_r\eta})\circ x\).
		Choose an orthonormal basis \(\{f_j\}_{j=1}^{d_m}\) of \(\mathcal{H}^0_{b,m}(X)\) with respect to the decomposition (\ref{Eq0}) and define a CR map \(\map{\Psi_m}{X}{\C^{d_m}}\), \(\Psi_m(x)=(f_1(x),\ldots,f_{d_m}(x))\). By construction the map \(\Psi_m\) is \(\R\)-equivariant, that is
		\[\Psi_m(\eta\cdot x)=(e^{i\nu_1\eta}f_1(x),\ldots,e^{i\nu_{d_m}\eta}f_{d_m}(x))\]
		for some \(\nu_1,\ldots,\nu_{d_m}\in\R\). Applying the embedding theorem in \cite{HHL16}  
		there exist \(m_1,\ldots,m_{\tilde{N}}\in\mathbb N\) such that the map \(\Phi:=(\Psi_{m_1},\ldots,\Psi_{m_{\tilde{N}}}):X\rightarrow \Complex^N\) is a CR embedding where \(N\in\mathbb N\) is some positive integer. Since \(\Psi_m\) is \(\R\)-equivariant the same is true for \(\Phi\) which completes the proof.
\end{proof}

Let $(X,T^{1,0}X)$ be a compact connected CR manifold equipped with a transversal CR $\R$-action. Assume that $X$ admits a rigid positive CR line bundle $L$.
From Theorem~\ref{thm:mainthm} and Corollary~\ref{c-gue170811}, we see that the $\Real$-action comes from a CR torus action \(T^d\curvearrowright X\) denoted by
$(e^{i\theta_1},\ldots,e^{i\theta_d})$. It should be mentioned that CR torus action means that \(T^d\) acts by CR automorphisms. 
  In this work, we need

\begin{theorem}\label{t-gue170811}
With the assumptions and notations above, we can find local CR rigid trivializations of $L$ defined on $D_j$, $j=1,\ldots,N$, such that $X=\bigcup^N_{j=1}D_j$, and
\begin{equation}\label{e-gue170811y}
D_j=\bigcup_{(e^{i\theta_1},\ldots,e^{i\theta_d})\in T^d}(e^{i\theta_1},\ldots,e^{i\theta_d})\circ D_j,\ \ j=1,2,\ldots,N,
\end{equation}
where $(e^{i\theta_1},\ldots,e^{i\theta_d})\circ D_j=\set{(e^{i\theta_1},\ldots,e^{i\theta_d})\circ x;\, x\in D_j}$.
\end{theorem}

\begin{proof}
Fix $p\in X$. Assume first that
\begin{equation}\label{e-gue170813}
\mbox{$(e^{i\theta_1},1,\ldots,1)\circ p\neq p$, for some $\theta_1\in]0,2\pi[$.}
\end{equation}
Put
\[\begin{split}
&A:=\{\lambda\in[0,2\pi];\, \mbox{We can find a neighborhood $W$ of $p$ and $\varepsilon>0$ such that there is}\\
&\mbox{ a local CR rigid trivializing section $s$ defined on
$\bigcup_{\theta_1\in[0,\lambda+\varepsilon[}(e^{i\theta_1},1,\ldots,1)\circ W$}\}.
\end{split}\]
It is clear that $A$ is a non-empty open set in $[0,2\pi]$. We claim that $A$ is closed. Let $\lambda_0$ be a limit point of $A$. Consider the point $q:=(e^{i\lambda_0},1,\ldots,1)\circ p$.
From \eqref{e-gue170813}, it is not difficult to see that $\frac{\pr}{\pr\theta_1}\neq0$ at $q$. We take local coordinates $x=(x_1,\ldots,x_{2n-1})$ defined on some neighborhood \[D=\set{x=(x_1,\ldots,x_{2n-1});\, \abs{x_j}<4\delta, j=1,\ldots,2n-1},\  \ \delta>0,\]
of $q$ so that $x(q)=0$, $\frac{\pr}{\pr\theta_1}=\frac{\pr}{\pr x_{2n-1}}$ on $D$. Let
\[D_0=\set{x=(x_1,\ldots,x_{2n-1});\, \abs{x_j}<\delta, j=1,\ldots,2n-1}.\]
Then, $D_0$ is an open neighborhood of $q$ and $(e^{i\theta_1},1,\ldots,1)\circ x=(x_1,\ldots,x_{2n-1}+\theta_1)$, for all $x\in D_0$, $\theta_1\in[0,\delta]$.
It is clear that for some $\delta>\varepsilon_1>0$, $\varepsilon_1$ small, there is a local CR trivializing section $s_1$ defined on
\[\bigcup_{\theta_1\in]\lambda_0-\varepsilon_1,\lambda_0+\varepsilon_1[}(e^{i\theta_1},1,\ldots,1)\circ\hat W\Subset D_0,\]
where $\hat W$ is a small neighborhood of $p$. Since $\lambda_0$ is a limit point of $A$, we can find a local CR trivializing section $\Td s$ defined on
\[\bigcup_{\theta_1\in[0,\lambda_0-\frac{\varepsilon_1}{4}[}(e^{i\theta_1},1,\ldots,1)\circ\Td W,\]
where $\Td W$ is a small neighborhood of $p$. Since $L$ is rigid, $\Td s=gs_1$ on
\[\bigcup_{\theta_1\in]\lambda_0-\varepsilon_1, \lambda_0-\frac{\varepsilon_1}{4}[}(e^{i\theta_1},1,\ldots,1)\circ\Bigr(\hat W\bigcap\Td W\Bigr)\]
for some rigid CR function $g$.
Let
\[W=\set{(e^{-i\lambda_0},1,\ldots,1)\circ x;\, x=(x_1,\ldots, x_{2n-1})\in D_0, \abs{x_j}<\gamma, j=1,\ldots,2n-1},\]
where $0<\gamma<\delta$ is a small constant so that $W\Subset \Bigr(\hat W\bigcap\Td W\Bigr)$. We consider $g$ as a function on
\[\bigcup_{\theta_1\in]\lambda_0-\varepsilon_1, \lambda_0-\frac{\varepsilon_1}{4}[}(e^{i\theta_1},1,\ldots,1)\circ W.\]
We claim that $g$ is independent of $x_{2n-1}$. Fix
\[x=(x_1,\ldots,x_{2n-1})\in \bigcup_{\theta_1\in]\lambda_0-\varepsilon_1, \lambda_0-\frac{\varepsilon_1}{4}[}(e^{i\theta_1},1,\ldots,1)\circ W\Subset D_0.\]
We have $g(x)=g((e^{i(x_{2n-1}+\frac{\varepsilon_1}{2})},1,\ldots,1)\circ (x_1,\ldots,x_{2n-2},-\frac{\varepsilon_1}{2}))$. Note that
\[(x_1,\ldots,x_{2n-2},-\frac{\varepsilon_1}{2})\in \bigcup_{\theta_1\in]\lambda_0-\varepsilon_1, \lambda_0-\frac{\varepsilon_1}{4}[}(e^{i\theta_1},1,\ldots,1)\circ W.\]
In view of Theorem~\ref{thm:mainthm}, we see that $\ol{\gamma(\R)}$ is the torus $T^d$, we can find a sequence of real numbers $t_j$, $j=1,2,\ldots$, such that
$t_j\circ (x_1,\ldots,x_{2n-2},-\frac{\varepsilon_1}{2})\To (e^{i(x_{2n-1}+\frac{\varepsilon_1}{2})},1,\ldots,1)\circ(x_1,\ldots,x_{2n-1},-\frac{\varepsilon_1}{2})$ as $j\To\infty$
and by Theorem~\ref{thm:steenrod},
\begin{equation}\label{e-gue170814}
g(x)=\lim_{j\To\infty}g(t_j\circ (x_1,\ldots,x_{2n-2},-\frac{\varepsilon_1}{2}))=g(x_1,\ldots,x_{2n-2},-\frac{\varepsilon_1}{2}).
\end{equation}
The claim follows. Hence, we can extend $g$ to $\bigcup_{\theta_1\in[\lambda_0-\frac{\varepsilon_1}{4}, \lambda_0+\frac{\varepsilon_1}{8}[}(e^{i\theta_1},1,\ldots,1)\circ W$ by
\begin{equation}\label{e-gue170813II}
\begin{split}
g:\bigcup_{\theta_1\in[\lambda_0-\frac{\varepsilon_1}{4}, \lambda_0+\frac{\varepsilon_1}{8}[}(e^{i\theta_1},1,\ldots,1)\circ W&\To\Complex,\\
x&\mapsto g(x_1,\ldots,x_{2n-2},-\frac{\varepsilon_1}{2}).
\end{split}
\end{equation}

Put $s=\Td s$ on $\bigcup_{\theta_1\in[0,\lambda_0-\frac{\varepsilon_1}{4}[}(e^{i\theta_1},1,\ldots,1)\circ W$ and  $s=gs_1$ on
\[\bigcup_{\theta_1\in]\lambda_0-\varepsilon_1, \lambda_0+\frac{\varepsilon_1}{8}[}(e^{i\theta_1},1,\ldots,1)\circ W.\]
It is straightforward to check that $s$ is well-defined as a local CR rigid trivializing section on $\bigcup_{\theta_1\in[0,\lambda_0+\frac{\varepsilon_1}{8}[}(e^{i\theta_1},1,\ldots,1)\circ W$.
Thus, $\lambda_0\in A$ and hence $A=[0,2\pi]$.

Now, assume that $(e^{i\theta_1},1,\ldots,1)\circ p=p$, for all $\theta_1\in[0,2\pi]$. Let $s$ be a local CR rigid trivializing section of $L$ defined on an open set $U$ of $p$. Since $(e^{i\theta_1},1,\ldots,1)\circ p=p$, for all $\theta_1\in[0,2\pi]$, we can find a small open set $W$ of $p$ with
\[\bigcup_{\theta_1\in[0,2\pi[}(e^{i\theta_1},1,\ldots,1)\circ W\subset U.\]
We conclude that there is a local CR rigid trivializing section of $L$ defined on
\[\bigcup_{\theta_1\in[0,2\pi]}(e^{i\theta_1},1,\ldots,1)\circ W.\]

Assume that $(1,e^{i\theta_2},1,\ldots,1)\circ p\neq p$, for some $\theta_2\in]0,2\pi[$. Put
\[\begin{split}
&B:=\{\lambda\in[0,2\pi];\, \mbox{We can find a neighborhood $W$ of $p$ and $\varepsilon>0$ such that there is}\\
&\mbox{ a local CR rigid trivializing section $s$ defined on
$\bigcup_{\theta_1\in[0,2\pi], \theta_2\in[0,\lambda+\varepsilon[}(e^{i\theta_1},e^{i\theta_2},\ldots,1)\circ W$}\}.
\end{split}\]
We can repeat the procedure above with minor change and conclude that $B=[0,2\pi]$. Assume that $(1,e^{i\theta_2},1,\ldots,1)\circ p=p$, $\forall\theta_2\in[0,2\pi]$.
It is clear that we can find a neighborhood $W$ of $p$ such that there is a local CR rigid trivializing section $s$ defined on
$\bigcup_{\theta_1\in[0,2\pi], \theta_2\in[0,2\pi]}(e^{i\theta_1},e^{i\theta_2},\ldots,1)\circ W$.
Continuing in this way, we conclude that for every $p\in X$,
there is a local CR rigid trivializing section $s$ of $L$ defined on
\[\bigcup_{(e^{i\theta_1},\ldots,e^{i\theta_d})\in T^d}(e^{i\theta_1},\ldots,e^{i\theta_d})\circ W,\]
where $W$ is an open set of $p$. Since $X$ is compact, the theorem follows.
\end{proof}

Theorem~\ref{t-gue170811} tells us that $L$ is torus invariant. From Theorem~\ref{t-gue170811}, we deduce

\begin{corollary}\label{c-gue170814}
With the assumptions and notations above, we can find local CR rigid trivializations $D_j$ of $L$, $j=1,\ldots,N$, such that $X=\bigcup^N_{j=1}D_j$, and
\begin{equation}\label{e-gue170811yq}
D_j=\bigcup_{t\in\Real}t\circ D_j,\ \ j=1,2,\ldots,N,
\end{equation}
where $t\circ D_j=\set{t\circ x;\, x\in D_j}$.
\end{corollary}

We also need

\begin{lemma}\label{l-gue170815}
With the assumptions and notations above, let $h^L$ be any rigid Hermitian fiber metric of $L$. Then, $h^L$ is $T^d$-invariant.
\end{lemma}

\begin{proof}
By Theorem~\ref{t-gue170811}, we can find local CR trivializations $D_j$ of $L$, $j=1,\ldots,N$, such that $X=\bigcup^N_{j=1}D_j$, and
\[D_j=\bigcup_{(e^{i\theta_1},\ldots,e^{i\theta_d})\in T^d}(e^{i\theta_1},\ldots,e^{i\theta_d})\circ D_j,\ \ j=1,2,\ldots,N,\]
where $(e^{i\theta_1},\ldots,e^{i\theta_d})\circ D_j=\set{(e^{i\theta_1},\ldots,e^{i\theta_d})\circ x;\, x\in D_j}$. For each $j=1,\ldots,N$, let $s_j$ be a local rigid CR trivialization
of $L$ on $D_j$, $\abs{s_j}^2_{h^{L_j}}=e^{-2\phi_j}$. Then, $\phi_j$ is $\Real$-invariant. Fix $(e^{i\theta_1},\ldots,e^{i\theta_d})\in T^d$ and $x\in D_j$. In view of Theorem~\ref{thm:mainthm}, we see that $\ol{\gamma(\R)}$ is the torus $T^d$ and we can find a sequence of $\Real$-action $t_k$, $k=1,2,\ldots$, such that $t_k\circ x\To (e^{i\theta_1},\ldots,e^{i\theta_d})\circ x$ as $k\To+\infty$ and by Theorem~\ref{thm:steenrod}, we have
\[\phi_j(x)=\lim_{k\To\infty}\phi_j(t_k\circ x)=\phi_j((e^{i\theta_1},\ldots,e^{i\theta_d})\circ x).\]
Thus, $h^L$ is torus invariant. The lemma follows.
\end{proof}

\section{The operator $-iT$}\label{s-gue170814}

From now on, we let $(X,T^{1,0}X)$ be a compact connected CR manifold
of dimension $2n-1$, $n\geqslant2$, endowed with a locally free transversal CR $\Real$-action $\eta$, $\eta\in\Real$: $\eta: X\to X$, $x\mapsto \eta\circ x$, and
let $(L,h^L)$ be a rigid CR line bundle over $X$ and assume that  there is an open interval $I\subset\Real$, such that $R^L-2s\mathcal{L}$ is positive definite on $X$, for every $s\in I$, where $h^L$ is a rigid Hermitian metric on $L$ and $R^L$ is the curvature of $L$ induced by $h^L$. For simplicity, we assume that $]-\delta,\delta[\subset I$, where $\delta>0$. Hence $R^L$ is positive on $X$.  Let
$\langle\,\cdot\,|\,\cdot\,\rangle$ be the rigid Hermitian metric on $\Complex TX$ induced by $R^L$ such that
$$T^{1,0}X\perp T^{0,1}X,\:\: T\perp (T^{1,0}X\oplus T^{0,1}X),\:\:
\langle\,T\,|\,T\,\rangle=1.$$ The rigid Hermitian metric $\langle\,\cdot\,|\,\cdot\,\rangle$ on $\Complex TX$ induces a rigid Hermitian metric $\langle\,\cdot\,|\,\cdot\,\rangle$ on $\oplus^{n-1}_{j=1}T^{*0,j}X$. We denote by $dv_X$ the volume form induced by $\langle\,\cdot\,|\,\cdot\,\rangle$.
Let $(\,\cdot\,|\,\cdot\,)_k$ be the $L^2$ inner product on $\Omega^{0,q}(X,L^k)$ induced by $h^{L^k}$ and
$dv_X$ and let $\norm{\cdot}_k$ be the corresponding norm. Let $L^2_{(0,q)}(X,L^k)$ be the completion of $\Omega^{0,q}(X,L^k)$ with respect
to $(\,\cdot\,|\,\cdot\,)_k$. We extend $(\,\cdot\,|\,\cdot\,)_k$ to $L^2_{(0,q)}(X,L^k)$. Consider the operator
\[-iT: \Omega^{0,q}(X,L^k)\To\Omega^{0,q}(X,L^k)\]
and we extend $-iT$ to $L^2_{(0,q)}(X,L^k)$ space by
\[\begin{split}
&-iT: {\rm Dom\,}(-iT)\subset L^2_{(0,q)}(X,L^k)\To L^2_{(0,q)}(X,L^k),\\
&{\rm Dom\,}(-iT)=\set{u\in L^2_{(0,q)}(X,L^k);\, -iTu\in L^2_{(0,q)}(X,L^k)}.
\end{split}\]

\begin{theorem}\label{t-gue170817yc}
The operator $-iT: {\rm Dom\,}(-iT)\subset L^2_{(0,q)}(X,L^k)\To L^2_{(0,q)}(X,L^k)$ is self-adjoint.
\end{theorem}

\begin{proof}
Let $(-iT)^*: {\rm Dom\,}(-iT)^*\subset L^2_{(0,q)}(X,L^k)\To L^2_{(0,q)}(X,L^k)$ be the Hilbert adjoint of $-iT$ with respect to $(\,\cdot\,|\,\cdot\,)_k$.
Since $(\,\cdot\,|\,\cdot\,)_k$ is rigid, we have $(\,-iTu\,|\,v\,)_k=(\,u\,|\,-iTv\,)_k$, for all $u, v\in\Omega^{0,q}(X,L^k)$. From this observation, it is easy to see that
${\rm Dom\,}(-iT)^*\subset  {\rm Dom\,}(-iT)$ and $-iTu=(-iT)^*u$, for all $u\in{\rm Dom\,}(-iT)^*$. Now, fix $u\in {\rm Dom\,}(-iT)$. We want to show that $u\in{\rm Dom\,}(-iT)^*$ and  $-iTu=(-iT)^*u$. Let $g\in{\rm Dom\,}(-iT)$. By classical Friedrichs' lemma, we can find $g_j\in\Omega^{0,q}(X,L^k)$, $j=1,2,\ldots$, such that $\norm{g_j-g}_k\To0$ as $j\To\infty$ and $\norm{(-iTg_j)-(-iTg)}_k\To0$ as $j\To\infty$. Now,
\[
(\,u\,|\,-iTg\,)_k=\lim_{j\To\infty}(\,u\,|\,-iTg_j\,)_k=\lim_{j\To\infty}(\,-iTu\,|\,g_j\,)_k=(\,-iTu\,|\,g\,)_k.\]
Hence, $u\in{\rm Dom\,}(-iT)^*$ and  $-iTu=(-iT)^*u$. The theorem follows.
\end{proof}

From Theorem~\ref{thm:mainthm} and Corollary~\ref{c-gue170811}, we see that the $\Real$-action $\eta$ comes from a torus action \(T^d\curvearrowright X\) denoted by $(e^{i\theta_1},\ldots,e^{i\theta_d})$.  By Theorem~\ref{t-gue170811}, we see that $X$ can be covered by torus invariant trivializations.  By using these  torus invariant trivializations,  the torus action on $X$ lifts to $L^k$. In view of Theorem~\ref{t-gue170811} and Lemma~\ref{l-gue170815}, we see that $L$, $h^L$ and $R^L$ are torus invariant and hence the Hermitian metric $\langle\,\cdot\,|\,\cdot\,\rangle$ and the $L^2$ inner product $(\,\cdot\,|\,\cdot\,)_k$ are torus invariant.

Note that $T^d\in\operatorname{Aut}_{\operatorname{CR}}(X)$. As in the $S^1$-action case (see Section 2.3 in~\cite{HLM16}), for $u\in\Omega^{0,q}(X,L^k)$ and for any $(e^{i\theta_1},\ldots,e^{i\theta_d})\in T^d$ , we define
\begin{equation}\label{e-gue150508faII}
u((e^{i\theta_1},\ldots,e^{i\theta_d})\circ x):=(e^{i\theta_1},\ldots,e^{i\theta_d})^*u(x)\in\Omega^{0,q}(X,L^k),
\end{equation}
where
\[(e^{i\theta_1},\ldots,e^{i\theta_d})^*: T^{*0,q}_xX\To T^{*0,q}_{(e^{-i\theta_1},\ldots,e^{-i\theta_d})\circ x}X\]
is the pull-back map of $(e^{i\theta_1},\ldots,e^{i\theta_d})$. For every $(m_1,\ldots,m_d)\in\mathbb Z^d$, put
\[\begin{split}
\Omega^{0,q}_{m_1,\ldots,m_d}(X,L^k):&=\{u\in\Omega^{0,q}(X,L^k);\, u((e^{i\theta_1},\ldots,e^{i\theta_d})\circ x)=e^{im_1\theta_1+\cdots+im_d\theta_d}u(x),\\
&\quad\forall(e^{i\theta_1},\ldots,e^{i\theta_d})\in T^d\}.
\end{split}\]
Let $L^2_{(0,q),m_1,\ldots,m_d}(X,L^k)$ be the $L^2$ completion of $\Omega^{0,q}_{m_1,\ldots,m_d}(X,L^k)$ with respect to $(\,\cdot\,|\,\cdot\,)_k$.
For $(m_1,\ldots,m_d)\in\mathbb Z^d$, let
\begin{equation}\label{e-gue170818}
Q^{(q)}_{m_1,\ldots,m_d,k}:L^2_{(0,q)}(X, L^k)\To L^2_{(0,q),m_1,\ldots,m_d}(X, L^k)
\end{equation}
be the orthogonal projection with respect to $(\,\cdot\,|\,\cdot\,)_k$. For $m=(m_1,\ldots,m_d)\in\mathbb Z^d$, denote $\abs{m}=\sqrt{m^2_1+\cdots+m^2_d}$.
By elementary Fourier analysis, it is straightforward to see that for every
$u\in\Omega^{0,q}(X,L^k)$,
\begin{equation}\label{e-gue170818I}
\begin{split}
&\mbox{$\lim\limits_{N\To\infty}\sum\limits_{m=(m_1,\ldots,m_d)\in\mathbb Z^d, \abs{m}\leq N}Q^{(q)}_{m_1,\ldots,m_d,k}u\To u$ in $C^\infty$ Topology},\\
&\sum_{m=(m_1,\ldots,m_d)\in\mathbb Z^d, \abs{m}\leq N}\norm{Q^{(q)}_{m_1,\ldots,m_d,k}u}_k^2\leq\norm{u}_k^2,\ \ \forall N>0.
\end{split}
\end{equation}
Thus, for every $u\in L^2_{(0,q)}(X, L^k)$,
\begin{equation}\label{e-gue170818II}
\begin{split}
&\mbox{$\lim\limits_{N\To\infty}\sum\limits_{m=(m_1,\ldots,m_d)\in\mathbb Z^d, \abs{m}\leq N}Q^{(q)}_{m_1,\ldots,m_d,k}u\To u$
in $L^2_{(0,q)}(X,L^k)$},\\
&\sum_{m=(m_1,\ldots,m_d)\in\mathbb Z^d, \abs{m}\leq N}\norm{Q^{(q)}_{m_1,\ldots,m_d,k}u}^2_k\leq\norm{u}^2_k,\ \ \forall N>0.
\end{split}
\end{equation}
For every $j=1,\ldots,d$, let $T_j$ be the operator on $\Omega^{0,q}(X)$ given by
\[(T_ju)(x)=\frac{\partial}{\partial\theta_j}u((1,\ldots,1,e^{i\theta_j},1,\ldots,1)\circ x)|_{\theta_j=0},\ \ \forall u\in\Omega^{0,q}(X).\]
Since $L$ is torus invariant, we can also define $T_ju$ in the standard way, for every $u\in\Omega^{0,q}(X,L^k)$, $j=1,\ldots,d$.
Note that $T_j$ can be zero at some point of $X$. Since the $\Real$-action $\eta$ comes from $T^d$, there exist real numbers $\beta_j\in\Real$, $j=1,\ldots,d$, such that
\begin{equation}\label{e-gue170828}
T=\beta_1T_1+\cdots+\beta_dT_d.
\end{equation}
Using the following remark we can assume that the \(\beta_j\)'s  in (\ref{e-gue170828}) are linear independent over \(\mathbb Q\).
\begin{remark}\label{lem:linearindependent}
	Assume that $\beta_1,\ldots,\beta_d$ in (\ref{e-gue170828}) are linear dependent over $\mathbb Q$. Without loss of generality, we may assume that $\beta_1,\ldots,\beta_p$ are linear independent over $\mathbb Q$, where $1\leq p<d$, and
	\begin{equation}\label{e-gue170830I}
	\beta_j=\sum^p_{\ell=1}r_{j,\ell}\beta_\ell,\ \ j=p+1,\ldots,d,
	\end{equation}
	where $r_{j,\ell}\in\mathbb Q$, for every $j=p+1,\ldots,d$, $\ell=1,\ldots,p$. Consider the new Torus action on \(X\):
	\[\begin{split}
	x&\mapsto (e^{i\theta_1},\ldots,e^{i\theta_p})\cdot x:=(e^{iN\theta_1},\ldots,e^{iN\theta_p},e^{iN\sum^p_{\ell=1}r_{p+1,\ell}\theta_\ell},\ldots,e^{iN\sum^p_{\ell=1}r_{d,\ell}\theta_\ell})\circ x,
	\end{split}\]
	where $N\in\mathbb N$ with $r_{j,\ell}|N$, for every $j=p+1,\ldots,d$, $\ell=1,\ldots,p$. For every $j=1,\ldots,p$, let $\hat T_j$ be the operator on $C^\infty(X)$ given by
	\[(\hat T_ju)(x)=\frac{\partial}{\partial\theta_j}u((1,\ldots,1,e^{i\theta_j},1,\ldots,1)\cdot x)|_{\theta_j=0},\ \ \forall u\in C^\infty(X).\]
	It is easy to check that the $\Real$-action $\eta$ comes from the new torus action
	$(e^{i\theta_1},\ldots,e^{i\theta_p})$ and
	\[T=\frac{\beta_1}{N}\hat T_1+\cdots+\frac{\beta_p}{N}\hat T_p.\]
	Note that $\frac{\beta_1}{N},\ldots, \frac{\beta_p}{N}$ are linear independent over $\mathbb Q$. Hence, without loss of generality
	we may assume that $\beta_1,\ldots, \beta_d$ are linear independent over $\mathbb Q$.
\end{remark}

\begin{lemma}\label{l-gue171004}
Fix  $(m_1,\ldots,m_d)\in\mathbb Z^d$. Then,  $L^2_{(0,q),m_1,\ldots,m_d}(X,L^k)\neq0$.
\end{lemma}

\begin{proof}
It is straightforward to see that we can find $\gamma_1\in\mathbb Q,\ldots,\gamma_d\in\mathbb Q$ such that the vector field
\[T_0:=\gamma_1T_1+\cdots+\gamma_dT_d\]
induces a transversal CR $S^1$ action $e^{i\theta}$ on $X$ with $X_{{\rm reg\,}}:=\set{x\in X;\, e^{i\theta}\circ x\neq x, \forall\theta\in]0,2\pi[}\neq\emptyset$ and
\begin{equation}\label{e-gue171004hyc}
\hat\Omega^{0,q}_{m_\gamma}(X,L^k):=\set{u\in\Omega^{0,q}(X,L^k);\, T_0u=im_\gamma u}=\Omega^{0,q}_{m_1,\ldots,m_d}(X,L^k),
\end{equation}
where $m_\gamma:=m_1\gamma_1+\cdots+m_d\gamma_d$. Fix $p\in X_{{\rm reg\,}}$ and let $x=(x_1,\ldots,x_{2n+1})=(x',x_{2n+1})$ be local coordinates of $X$ centered at $p=0$ defined on
\[D=\set{x=(x_1,\ldots,x_{2n+1})\in\Real^{2n+1};\, \abs{x'}<\delta, \abs{x_{2n+1}}<\varepsilon}\] such that $T_0=\frac{\pr}{\pr x_{2n+1}}$, where $\delta>0$, $\varepsilon>0$ are constants and $x'=(x_1,\ldots,x_{2n})$. Let $s$ be a local CR rigid trivializing section of $L$ on $D$.  It is not difficult to see that there is a small open set $D_0\Subset D$ of $p$ such that for all $(x',0)\in D_0$, we have
\begin{equation}\label{e-gue171004y}
e^{i\theta}\circ (x',0)\notin D_0,\ \ \forall\theta\in]\varepsilon, 2\pi-\varepsilon].
\end{equation}
Let $\chi(x)\in C^\infty(D_0)$ with
\begin{equation}\label{e-gue171004yI}
\int\chi(x',x_{2n+1})dx_{2n+1}\neq0.
\end{equation}
Let $u(x):=s^k(x)\otimes\chi(x)e^{im_\gamma x_{2n+1}}\in C^\infty(X,L^k)$. From \eqref{e-gue171004y} and \eqref{e-gue171004yI} we can check that
\[\frac{1}{2\pi}\int^{2\pi}_0u(e^{i\theta}\circ (x',0))e^{-im_\gamma\theta}d\theta=\frac{1}{2\pi}s^k(x)\otimes \int\chi(x',x_{2n+1})dx_{2n+1}\neq0.\]
Since $\frac{1}{2\pi}\int^{2\pi}_0u(e^{i\theta}\circ x)e^{-im_\gamma\theta}d\theta\in\hat\Omega^{0,q}_{m_\gamma}(X,L^k)$, we deduce that
$\hat\Omega^{0,q}_{m_\gamma}(X,L^k)\neq\set{0}$. From this observation and \eqref{e-gue171004hyc}, the lemma follows.
\end{proof}

We need

\begin{lemma}\label{l-gue170828}
 Assume that $\beta_1,\ldots,\beta_d$ in (\ref{e-gue170828})  are linear independent over $\mathbb Q$. Given $p:=(p_1,\ldots,p_d)\in\mathbb Z^d$ we have that $p_\beta:=\sum^d_{j=1}p_j\beta_j$ is an eigenvalue of $-iT$ and the corresponding eigenspace is
$L^2_{(0,q),p_1,\ldots,p_d}(X,L^k)$.
\end{lemma}

\begin{proof}
Set
\[E_{p_\beta}:=\set{u\in{\rm Dom\,}(-iT);\, -iTu=p_\beta u}.\]
Given $u\in\Omega^{0,q}_{p_1,\ldots,p_d}(X,L^k)$ it is easy to check that
\[-iTu=-i\sum^d_{j=1}\beta_jT_ju=\sum^d_{j=1}\beta_jp_ju=p_\beta u\]
and hence $u\in E_{p_\beta}$. We obtain $\Omega^{0,q}_{p_1,\ldots,p_d}(X,L^k)\subset E_{p_\beta}$. Let $g\in L^2_{(0,q),p_1,\ldots,p_d}(X,L^k)$.
Take $g_j\in\Omega^{0,q}_{p_1,\ldots,p_d}(X,L^k)$, $j=1,2,\ldots$, such that $g_j\To g$ in $L^2_{(0,q)}(X,L^k)$ as $j\To+\infty$. Since $-iTg_j=p_\beta g_j$, for every $j$, we deduce that $-iTg=p_\beta g$ in the sense of distribution. Thus, $g\in E_{p_\beta}$. We have proved that $L^2_{(0,q),p_1,\ldots,p_d}(X,L^k)\subset E_{p_\beta}$. 

We claim that $L^2_{(0,q),p_1,\ldots,p_d}(X,L^k)\supset E_{p_\beta}$.
If the claim is not true, we can find a $u\in E_{p_\beta}$, $\norm{u}_k=1$, such that
\begin{equation}\label{e-gue170828y}
u\perp L^2_{(0,q),p_1,\ldots,p_d}(X,L^k).
\end{equation}
From \eqref{e-gue170818II}, we have
\begin{equation}\label{e-gue170828yI}
\mbox{$\lim\limits_{N\To\infty}\sum\limits_{m=(m_1,\ldots,m_d)\in\mathbb Z^d, \abs{m}\leq N}Q^{(q)}_{m_1,\ldots,m_d,k}u\To u$
in $L^2_{(0,q)}(X,L^k)$}.
\end{equation}
Note that for every $m=(m_1,\ldots,m_d)\in\mathbb Z^d$, $Q^{(q)}_{m_1,\ldots,m_d,k}u$ is an eigenvector of $-iT$ with eigenvalue $\sum^d_{j=1}m_j\beta_j$ and since \(\beta_1,\ldots,\beta_d\) are linear independent over \(\mathbb Q\) we have \(p_\beta=\sum^d_{j=1}m_j\beta_j\) if and only if \((p_1,\ldots,p_d)=(m_1,\ldots,m_d)\). From this observations and \eqref{e-gue170828y}, we conclude that
\begin{equation}\label{e-gue170828yII}
(\,u\,|\,Q^{(q)}_{m_1,\ldots,m_d,k}u\,)_k=0,\ \ \forall (m_1,\ldots,m_d)\in\mathbb Z^d.
\end{equation}
From \eqref{e-gue170828yII}, we see that for every $N\in\mathbb N$, we have
\begin{equation}\label{e-gue170828yIII}
\begin{split}
&\norm{u-\sum\limits_{m=(m_1,\ldots,m_d)\in\mathbb Z^d, \abs{m}\leq N}Q^{(q)}_{m_1,\ldots,m_d,k}u}^2_k\\
&=\norm{u}^2_k+\sum\limits_{m=(m_1,\ldots,m_d)\in\mathbb Z^d, \abs{m}\leq N}\norm{Q^{(q)}_{m_1,\ldots,m_d,k}u}^2_k.
\end{split}
\end{equation}
From \eqref{e-gue170828yIII} and \eqref{e-gue170828yI}, we get a contradiction. The lemma follows.
\end{proof}

Let ${\rm Spec\,}(-iT)$ denote the spectrum of $-iT$. We can now prove

\begin{theorem}\label{t-gue170828}
${\rm Spec\,}(-iT)$ is countable
and every element in ${\rm Spec\,}(-iT)$ is an eigenvalue of $-iT$. Moreover, for every $\alpha\in{\rm Spec\,}(-iT)$, we can find
\[(m_1,\ldots,m_d)\in\mathbb Z^d\]
such that $\alpha=\sum^d_{j=1}\beta_jm_j$, where $\beta_1\in\Real,\ldots,\beta_d\in\Real$, are as in \eqref{e-gue170828}.
\end{theorem}

\begin{proof}
Let $A=\set{\alpha=\sum^d_{j=1}m_j\beta_j;\, (m_1,\ldots,m_d)\in\mathbb Z^d}$.
From Lemma~\ref{l-gue170828}, we see that $A\subset{\rm Spec\,}(-iT)$ and every element in $A$ is an eigenvalue of $-iT$.
For a Borel set $B$ of $\Real$, we denote by $E(B)$ the spectral projection of $-iT$ corresponding to the set $B$, where $E$ is the spectral measure of $-iT$.
Fix a Borel set $B$ of $\Real$ with $B\bigcap A=\emptyset$ and let $g\in{\rm Range\,}E(B)\subset L^2_{(0,q)}(X,L^k)$. Since $B\bigcap A=\emptyset$ and by Lemma~\ref{l-gue170828}, we see that
\begin{equation}\label{e-gue170828yIIa}
(\,g\,|\,Q^{(q)}_{m_1,\ldots,m_d,k}g\,)_k=0,\ \ \forall (m_1,\ldots,m_d)\in\mathbb Z^d.
\end{equation}
From \eqref{e-gue170828yIIa} and \eqref{e-gue170818II}, we get
\[\begin{split}
&\norm{g-\sum\limits_{m=(m_1,\ldots,m_d)\in\mathbb Z^d, \abs{m}\leq N}Q^{(q)}_{m_1,\ldots,m_d,k}g}^2_k\\
&=\norm{g}^2_k+\sum\limits_{m=(m_1,\ldots,m_d)\in\mathbb Z^d, \abs{m}\leq N}\norm{Q^{(q)}_{m_1,\ldots,m_d,k}g}^2_k\\
&\To0\ \ \mbox{as $N\To\infty$}.
\end{split}\]
Hence, $g=0$. We have proved that $A={\rm Spec\,}(-iT)$. The theorem follows.
\end{proof}

We will prove now that
\[\mathcal{H}^0_{b,\leq \lambda}(X,L^k):=\bigoplus_{\alpha\in{\rm Spec\,}(-iT), \abs{\alpha}\leq \lambda}\mathcal{H}^0_{b,\alpha}(X,L^k)\]
is finite dimensional, that is \(\mathcal{H}^0_{b,\alpha}(X,L^k)=0\) for almost every \(\alpha\in{\rm Spec\,}(-iT), \abs{\alpha}\leq \lambda\).
\begin{lemma}\label{lem:Hdimfinite}
	We have that \(\dim\mathcal{H}^0_{b,\leq \lambda}(X,L^k)<\infty\) and hence that \(\mathcal{H}^0_{b,\leq \lambda}(X,L^k)\) is closed.
\end{lemma}
\begin{proof}
Consider the operator $\Box_{b,k}:=\ol{\pr}^{*}_{b}\ddbar_b: C^\infty(X,L^k)\To C^\infty(X,L^k)$ where \[\ol{\pr}^{*}_b:\Omega^{0,1}(X,L^k)\To C^\infty(X,L^k)\] is the formal adjoint of $\ddbar_b$ with respect to $(\,\cdot\,|\,\cdot\,)_k$ (see Section \ref{s-gue170828} for a detailed description).   
 Consider
\[\triangle_k:=\Box_{b,k}-T^2: C^\infty(X,L^k)\To C^\infty(X,L^k)\]
and we extend $\triangle_k$ to $L^2$ space by $\triangle_k: {\rm Dom\,}\triangle_k\subset L^2(X,L^k)\To L^2(X,L^k)$, ${\rm Dom\,}\triangle_k=\set{u\in L^2(X,L^k);\, \triangle_ku\in L^2(X,L^k)}$ and $\triangle_k=(\Box_{b,k}-T^2)u$, for every $u\in{\rm Dom\,}\triangle_k$. Let $\sigma_{\Box_{b,k}}$ and $\sigma_{T^2}$ denote the
principal symbols of $\Box_{b,k}$ and $T^2$ respectively.
It is well-known (see the discussion after Proposition 2.3 in~\cite{Hsiao08}) that there is a constant $C>0$ such that
\begin{equation}\label{e-gue171012}
\sigma_{\Box_{b,k}}(x,\xi)+\abs{\langle\,\xi\,|\,\omega_0(x)\,\rangle}^2\geq C\abs{\xi}^2,\ \ \forall (x,\xi)\in T^*X.
\end{equation}
Moreover, it is easy to see that
\begin{equation}\label{e-gue171012I}
\sigma_{T^2}(x,\xi)=-\abs{\langle\,\xi\,|\,\omega_0(x)\,\rangle}^2,\ \ \forall (x,\xi)\in T^*X.
\end{equation}
From \eqref{e-gue171012} and \eqref{e-gue171012I}, we deduce that
$\triangle_k$ is elliptic. As a consequence ${\rm Spec\,}(\triangle_k)$ is discrete and every element in ${\rm Spec\,}(\triangle_k)$ is an eigenvalue of $\triangle_k$. For every $\mu\in{\rm Spec\,}(\triangle_k)$, put
$E_\mu(\triangle_k):=\set{u\in{\rm Dom\,}\triangle_k;\, \triangle_ku=\mu u}$.
For every $\lambda\geq0$, it is easy to see that
\begin{equation}\label{e-gue171004}
\mathcal{H}^0_{b,\leq\lambda}(X,L^k)\subset\oplus_{\mu\in{\rm Spec\,}(\triangle_k), \abs{\mu}\leq\lambda^2}E_\mu(\triangle_k).
\end{equation}
From \eqref{e-gue171004} and notice that ${\rm dim\,}E_\mu(\triangle_k)<+\infty$, for every $\mu\in{\rm Spec\,}(\triangle_k)$, the lemma follows.
\end{proof}

\section{Szeg\H{o} kernels and equivariant embedding theorems}\label{s-gue170828}

In this section, we will prove Theorem~\ref{t-gue150807} and Theorem~\ref{t-gue150807I}. We first recall some results in~\cite{Hsiao14}. We refer the reader to Section 2.2 in~\cite{HLM16} for some notations in semi-classical analysis used here. Let
\[\ol{\pr}^{*}_b:\Omega^{0,1}(X,L^k)\To C^\infty(X,L^k)\]
be the formal adjoint of $\ddbar_b$ with respect to $(\,\cdot\,|\,\cdot\,)_k$. Since $\langle\,\cdot\,|\,\cdot\,\rangle$ and $h$ are rigid, we can check that
\begin{equation}\label{e-gue150517}
\begin{split}
&T\ddbar^{*}_b=\ddbar^{*}_bT\ \ \mbox{on $\Omega^{0,1}(X,L^k)$, $q=1,2,\ldots,n-1$},\\
&\ddbar^{*}_b:\Omega^{0,1}_\alpha(X,L^k)\To C^\infty_\alpha(X,L^k),\ \ \forall\alpha\in{\rm Spec\,}(-iT),
\end{split}
\end{equation}
where $\Omega^{0,1}_\alpha(X,L^k)=\set{u\in\Omega^{0,1}(X,L^k);\, -iTu=\alpha u}$.
Put
\begin{equation}\label{e-gue150517I}
\Box_{b,k}:=\ol{\pr}^{*}_{b}\ddbar_b:C^\infty(X,L^k)\To C^\infty(X,L^k).
\end{equation}
From \eqref{e-gue150508d} and \eqref{e-gue150517}, we have
\begin{equation}\label{e-gue150517II}
\begin{split}
&T\Box_{b,k}=\Box_{b,k}T\ \ \mbox{on $C^\infty(X,L^k)$},\\
&\Box_{b,k}:C^\infty_\alpha(X,L^k)\To C^\infty_\alpha(X,L^k),\ \ \forall\alpha\in{\rm Spec\,}(-iT).
\end{split}
\end{equation}
Let $\Pi_k:L^2(X)\To{\rm Ker\,}\Box_{b,k}$ be the orthogonal projection
(the Szeg\H{o} projector).

\begin{definition}\label{d-gue130820m}
Let $A_k:L^2(X,L^k)\To L^2(X,L^k)$ be a continuous operator.
Let $D\Subset X$. We say that $\Box_{b,k}$ has $O(k^{-n_0})$
small spectral gap on $D$ with respect to $A_k$ if for every $D'\Subset D$,
there exist constants $C_{D'}>0$,  $n_0, p\in\mathbb N$, $k_0\in\mathbb N$,
such that for all $k\geq k_0$ and $u\in C^\infty_0(D',L^k)$,
we have
\[\norm{A_k(I-\Pi_k)u}_k\leq
C_{D'}\,k^{n_0}\sqrt{(\,(\Box_{b,k})^pu\,|\,u\,)_k}\,.\]
\end{definition}

Fix $\lambda>0$ and let $\Pi_{k,\leq\lambda}$ be as in \eqref{e-gue150806V}.

\begin{definition}\label{d-gue131205m}
Let $A_k:L^2(X,L^k)\To L^2(X,L^k)$ be a continuous operator.
We say that $\Pi_{k,\leq\lambda}$ is $k$-negligible away the diagonal with respect to $A_k$
on $D\Subset X$ if for any $\chi, \chi_1\in C^\infty_0(D)$ with $\chi_1=1$
on some neighborhood of ${\rm Supp\,}\chi$, we have
\[\big(\chi A_k(1-\chi_1)\big)\Pi_{k,\leq\lambda}\big(\chi A_k(1-\chi_1)\big)^*=
O(k^{-\infty})\ \ \mbox{ on $D$},\]
where $\big(\chi A_k(1-\chi_1)\big)^*:L^2(X,L^k)\To L^2(X,L^k)$
is the Hilbert space adjoint of $\chi A_k(1-\chi_1)$ with respect to $(\,\cdot\,|\,\cdot\,)$.
\end{definition}

Fix  $\delta>0$ and let $F_{k,\delta}$ be as in \eqref{e-gue150807I}. Let $s$ be a local rigid CR trivializing section of $L$ on an open set $D$ of $X$.
The localization of $F_{k,\delta}$ with respect to the trivializing rigid CR section $s$ is given by
\begin{equation}\label{e-gue170909}
F_{k,\delta,s}: L^2_{{\rm comp\,}}(D)\To L^2(D),\ \ F_{k,\delta,s}=U^{-1}_{k,s}F_{k,\delta}U_{k,s},
\end{equation}
where $U_{k,s}$ is as in \eqref{e-gue150806VI}. The following is well-known

\begin{theorem}[{\cite[Theorem 1.5]{Hsiao14}}]\label{t-gue150811}
With the notations and assumptions used above,
let $s$ be a local rigid CR trivializing section of $L$ on a canonical coordinate patch $D\Subset X$
with canonical coordinates $x=(z,\theta)=(x_1,\ldots,x_{2n-1})$,
$\abs{s}^2_{h^L}=e^{-2\Phi}$. Let $\delta>0$ be a constant so that
$R^L_{x}-2t\mathcal{L}_{x}$ is positive definite, for every $x\in X$ and
$\abs{t}\leq\delta$. Let $F_{k,\delta}$ be as in \eqref{e-gue150807I} and
let $F_{k, \delta,s}$ be the localized operator of $F_{k,\delta}$ given by \eqref{e-gue170909}.
Assume that:

{\rm (I)\,} $\Box_{b,k}$ has $O(k^{-n_0})$ small spectral gap on $D$ with respect to $F_{k,\delta}$.

{\rm (II)\,} $\Pi_{k,\leq\delta k}$ is $k$-negligible away the diagonal with
respect to $F_{k,\delta}$ on $D$.

{\rm (III)\,} $F_{k, \delta,s}-B_k=O(k^{-\infty}):
H^s_{{\rm comp\,}}(D)\To H^s_{{\rm loc\,}}(D)$, $\forall s\in\mathbb N_0$, where
\[B_k=\frac{k^{2n-1}}{(2\pi)^{2n-1}}\int e^{ik\langle x-y,\xi\rangle }\alpha(x,\xi,k)d\eta
+ O(k^{-\infty})\]
is a classical semi-classical pseudodifferential operator on $D$ of order $0$ with
\[\begin{split}&\mbox{$\alpha(x,\xi,k)\sim\sum_{j=0}^\infty\alpha_j(x,\xi)k^{-j}$
in $S^0_{{\rm loc\,}}(1;T^*D)$},\\
&\alpha_j(x,\xi)\in C^\infty(T^*D),\ \ j=0,1,\ldots,
\end{split}\]
and for every $(x,\xi)\in T^*D$, $\alpha(x,\xi,k)=0$
if $\big|\langle\,\xi\,|\,\omega_0(x)\,\rangle\big|> \delta$. Fix $D_0\Subset D$. Then
\begin{equation}\label{c1}
P_{k,\delta,s}(x,y)=\int e^{ik\varphi(x,y,t)}g(x,y,t,k)dt+O(k^{-\infty})\:\:
\text{on $D_0\times D_0$},
\end{equation}
where $\varphi(x,y,t)\in C^\infty( D\times D\times(-\delta,\delta))$ is as in \eqref{e-gue150807b} and
\[\begin{split}
&g(x,y,t,k)\in S^{n}_{{\rm loc\,}}(1;D\times D\times(-\delta,\delta))\cap C^\infty_0(D\times D\times(-\delta,\delta)),\\
&g(x,y,t,k)\sim\sum^\infty_{j=0}g_j(x,y,t)k^{n-j}\text{ in }S^{n}_{{\rm loc\,}}(1;D\times D\times(-\delta,\delta))
\end{split}\]
is as in \eqref{e-gue150807bI}, where $P_{k,\delta,s}$ is given by \eqref{e-gue150806VII}.
\end{theorem}

In view of Theorem~\ref{t-gue150811}, we see that to prove
Theorem~\ref{t-gue150807}, we only need to prove  that
{\rm (I)\,},  {\rm (II)\,} and  {\rm (III)\,} in Theorem~\ref{t-gue150811}
hold if $\delta>0$ is small enough.  By repeating the proof of Theorem 3.9 in~\cite{HLM16}, we see that {\rm (I)\,} holds. We only need to show that {\rm (II)\,} and  {\rm (III)\,} hold.

Recall that the $\Real$-action $\eta$ comes from a CR torus action \(T^d\curvearrowright X\) which we denote by $(e^{i\theta_1},\ldots,e^{i\theta_d})$ and $L$, $h^L$, $R^L$, the Hermitian metric $\langle\,\cdot\,|\,\cdot\,\rangle$ and the $L^2$ inner product $(\,\cdot\,|\,\cdot\,)_k$ are torus invariant. For every $j=1,\ldots,d$, let $T_j$ be the operator on $C^\infty(X)$ given by
\[(T_ju)(x)=\frac{\partial}{\partial\theta_j}u((1,\ldots,1,e^{i\theta_j},1,\ldots,1)\circ x)|_{\theta_j=0},\ \ \forall u\in C^\infty(X).\]
Since the $\Real$-action $\eta$ comes from $T^d$, there exist real numbers $\beta_j\in\Real$, $j=1,\ldots,d$, such that
\begin{equation}\label{e-gue170830}
T=\beta_1T_1+\cdots+\beta_dT_d.
\end{equation}
Using Remark \ref{lem:linearindependent} we can assume that \(\beta_1,\ldots,\beta_d\) are linear independent over \(\mathbb Q\).

Let $D\subset X$ be a canonical coordinate patch and let $x=(x_1,\ldots,x_{2n-1})$ be canonical coordinates on $D$. We identify $D$ with $W\times]-\varepsilon,\varepsilon[\subset\Real^{2n-1}$, where $W$ is some open set in $\Real^{2n-2}$ and $\varepsilon>0$. Until further notice, we work with canonical coordinates $x=(x_1,\ldots,x_{2n-1})$. Let $\xi=(\xi_1,\ldots,\xi_{2n-1})$ be the dual coordinates of $x$.
Let $s$ be a local rigid CR trivializing section of $L$ on $D$, $\abs{s}^2_{h^L}=e^{-2\Phi}$. Let $F_{k, \delta,s}$ be the localized operator of $F_{k,\delta}$ given by \eqref{e-gue170909}. Put
\begin{equation}\label{e-gue131209}
B_{k}=\frac{k^{2n-1}}{(2\pi)^{2n-1}}\int e^{ik\langle x-y,\xi\rangle}\tau_\delta(\xi_{2n-1})d\xi,
\end{equation}
where $\tau_\delta\in C^\infty_0((-\delta, \delta))$ is given by \eqref{e-gue160105}.

\begin{lemma}\label{l-gue131209}
We have
\[F_{k, \delta,s}-B_{k}=O(k^{-\infty}):H^s_{{\rm comp\,}}(D)\To H^s_{{\rm loc\,}}(D),\ \ \forall s\in\mathbb N_0.\]
\end{lemma}

\begin{proof}
We also write $y=(y_1,\ldots,y_{2n-1})$ to denote the canonical coordinates $x$. It is easy to see that on $D$,
\begin{equation}\label{e-gue131217}
\begin{split}
&F_{k, \delta,s}u(y)\\
&=\sum_{(m_1,\ldots,m_d)\in\mathbb Z^d}\tau_\delta\Bigr(\frac{\sum^d_{j=1}m_j\beta_j}{k}\Bigr)e^{i(\sum^d_{j=1}m_j\beta_j) y_{2n-1}}\\
&\quad\times\int_{T^d}e^{-(im_1\theta_1+\cdots+im_d\theta_d)}
u((e^{i\theta_1},\ldots,e^{i\theta_d})\circ y')dT_d,\ \ \forall u\in C^\infty_0(D),
\end{split}
\end{equation}
where $y'=(y_1,\ldots,y_{2n-2},0)$, $dT_d=(2\pi)^{-d}d\theta_1\cdots d\theta_d$ and $\beta_1\in\Real,\ldots,\beta_d\in\Real$ are as in \eqref{e-gue170830}.
Recall that $\beta_1,\ldots, \beta_d$ are linear independent over $\mathbb Q$.
Fix $D'\Subset D$ and let $\chi(y_{2n-1})\in C^\infty_0(]-\varepsilon,\varepsilon[)$ such that $\chi(y_{2n-1})=1$ for every $(y',y_{2n-1})\in D'$. Let $R_{k}:C^\infty_0(D')\To C^\infty(D')$ be the continuous operator given by
\begin{equation}\label{e-gue131217I}
\begin{split}
&(R_ku)(x)=\\
&\frac{1}{2\pi}\sum_{(m_1,\ldots,m_d)\in\mathbb Z^d}\:
\int\limits_{T^d}e^{i\langle x_{2n-1}-y_{2n-1},\xi_{2n-1}\rangle+i(\sum^d_{j=1}m_j\beta_j)y_{2n-1}-im_1\theta_1-\cdots-im_d\theta_d}\\
&\times\tau_\delta\Bigr(\frac{\xi_{2n-1}}{k}\Big) (1-\chi(y_{2n-1}))u((e^{i\theta_1},\ldots,e^{i\theta_d})\circ x')dT_dd\xi_{2n-1}dy_{2n-1},
\end{split}
\end{equation}
where $u\in C^\infty_0(D')$. We claim that
\begin{equation}\label{e-gue170831}
R_k=O(k^{-\infty}): H^s_{{\rm comp\,}}(D')\To H^s_{{\rm loc\,}}(D'),\ \ \forall s\in\mathbb N_0.
\end{equation}
We only prove the claim \eqref{e-gue170831} for $s=0$. For any $s\in\mathbb N$, the proof is similar. Fix any $g\in C^\infty_0(D')$. By using integration by parts with respect to $y_{2n-1}$ and $\xi_{2n-1}$ several times, it is straightforward to check that for every $N\in\mathbb N$, there is a constant $C_N>0$ independent of $k$ such that
\begin{equation}\label{e-gue170831I}
\begin{split}
&\int_X\abs{R_ku}^2(x)g(x)dv_X(x)\\
&\leq C_Nk^{-2N}\Bigr(\sum_{(m_1,\ldots,m_d)\in\mathbb Z^d, (m_1,\ldots,m_d)\neq(0,\ldots,0)}\Bigr(\frac{1}{m_1\beta_1+\cdots+m_d\beta_d}\Bigr)^{2N}+1\Bigr)\\
&\times\int_X\int_{T^d}\abs{u((e^{i\theta_1},\ldots,e^{i\theta_d})\circ x')}^2g(x)dT_ddv_X(x).
\end{split}
\end{equation}
It is clear that $\int_{T^d}\abs{u((e^{i\theta_1},\ldots,e^{i\theta_d})\circ x')}^2g(x)dT_d=\int_{T^d}\abs{u((e^{i\theta_1},\ldots,e^{i\theta_d})\circ x)}^2g(x)dT_d$. From this observation, \eqref{e-gue170831I} and notice that $dv_X$ is $T^d$-invariant, we conclude that
\begin{equation}\label{e-gue170831II}
\begin{split}
&\int_X\int_{T^d}\abs{u((e^{i\theta_1},\ldots,e^{i\theta_d})\circ x')}^2g(x)dT_ddv_X(x)\\
&=\int_X\int_{T^d}\abs{u((e^{i\theta_1},\ldots,e^{i\theta_d})\circ x)}^2g(x)dT_ddv_X(x)\\
&\leq C\int_X\int_{T^d}\abs{u((e^{i\theta_1},\ldots,e^{i\theta_d})\circ x)}^2dT_ddv_X(x)\\
&\leq\int_X\abs{u(x)}^2dv_X(x),
\end{split}
\end{equation}
where $C>0$ is a constant independent of $k$ and $u$. From \eqref{e-gue170831II}, \eqref{e-gue170831I} and note that
\[\sum_{(m_1,\ldots,m_d)\in\mathbb Z^d}\Bigr(\frac{1}{m_1\beta_1+\cdots+m_d\beta_d})^{2N}<+\infty\]
if $N$ large enough, we get the claim \eqref{e-gue170831}.

Now, we claim that
\begin{equation}\label{e-gue131217III}
B_{k}+R_k=F_{k, \delta,s}\ \ \mbox{on $C^\infty_0(D')$}.
\end{equation}
Let $u\in C^\infty_0(D')$. From \eqref{e-gue131209} and Fourier inversion formula, it is straightforward to see that
\begin{equation}\label{e-gue131217IV}
\begin{split}
&B_{k}u(x)\\
&=\frac{1}{2\pi}\sum_{(m_1,\ldots,m_d)\in\mathbb Z^d}
\int e^{i\langle x_{2n-1}-y_{2n-1},\xi_{2n-1}\rangle}
\tau_\delta\Big(\frac{\xi_{2n-1}}{k}\Big)\chi(y_{2n-1})\\
&\quad\times e^{i(\sum^d_{j=1}m_j\beta_j)y_{2n-1}-im_1\theta_1-\cdots-im_d\theta_d}u((e^{i\theta_1},\ldots,e^{i\theta_d})\circ x')dT_ddy_{2n-1}d\xi_{2n-1}.
\end{split}
\end{equation}
From \eqref{e-gue131217IV} and \eqref{e-gue131217I}, we have
\begin{equation}\label{e-gue131217V}
\begin{split}
&(B_{k}+R_{k})u(x)\\
&=\frac{1}{2\pi}\sum_{(m_1,\ldots,m_d)\in\mathbb Z^d}
\int e^{i\langle x_{2n-1}-y_{2n-1},\xi_{2n-1}\rangle}
\tau_\delta\Big(\frac{\xi_{2n-1}}{k}\Big)\\
&\quad\times e^{i(\sum^d_{j=1}m_j\beta_j)y_{2n-1}-im_1\theta_1-\cdots-im_d\theta_d}u((e^{i\theta_1},\ldots,e^{i\theta_d})\circ x')dT_ddy_{2n-1}d\xi_{2n-1}.
\end{split}\end{equation}
Note that the following formula holds for every $\alpha\in\Real$,
\begin{equation}\label{dm}
\int e^{i\alpha y_{2n-1}}e^{-iy_{2n-1}\xi_{2n-1}}dy_{2n-1}=2\pi\delta_\alpha(\xi_{2n-1}),
\end{equation}
where the integral is defined as an oscillatory integral and $\delta_\alpha$ is the Dirac measure at $\alpha$.
Using \eqref{e-gue131217}, \eqref{dm} and the Fourier inversion formula,
 \eqref{e-gue131217V} becomes
\begin{equation}\label{e-gue131217VI}
\begin{split}
&(B_{k}+R_k)u(x)\\
&=\sum_{(m_1,\ldots,m_d)\in\mathbb Z^d}\tau_\delta\Big(\frac{\sum^d_{j=1}m_j\beta_j}{k}\Big)e^{i(\sum^d_{j=1}m_j\beta_j)x_{2n-1}}
\int_{T_d}e^{-im_1\theta_1-\cdots-im_d\theta_d}u((e^{i\theta_1},\ldots,e^{i\theta_d})\circ x')dT_d\\
&=F_{k, \delta,s}u(x).
\end{split}\end{equation}
From \eqref{e-gue131217VI}, the claim \eqref{e-gue131217III} follows.
From \eqref{e-gue131217III} and \eqref{e-gue170831}, the lemma follows.
\end{proof}

From Lemma~\ref{l-gue131209}, we see that the condition {\rm (III)\,} in Theorem~\ref{t-gue150811} holds.

\begin{lemma}\label{l-gue150813}
Let $D\subset X$ be a canonical coordinate patch of $X$.
Then, $\Pi_{k,\leq\delta k}$ is $k$-negligible away the diagonal with respect to $F_{k,\delta}$ on $D$.
\end{lemma}

\begin{proof}
Let $\chi, \chi_1\in C^\infty_0(D)$, $\chi_1=1$ on some neighbourhood of
${\rm Supp\,}\chi$. Let $u\in\mathcal{H}^0_{b,\leq k\delta}(X,L^k)$ with
$\norm{u}_k=1$. We can repeat the proof of Theorem 2.4 in~\cite{HL15} and deduce that there is a constant $C>0$ independent of $k$ and $u$ such that
\begin{equation}\label{e-gue150813}
\abs{u(x)}^2_{h^k}\leq Ck^n,\ \ \forall x\in X.
\end{equation}
Let $x=(x_1,\ldots,x_{2n-1})=(x',x_{2n-1})$ be canonical coordinates on $D$ and let $s$ be a rigid CR trivializing section of $L$ on $D$, $\abs{s}^2_{h^L}=e^{-2\Phi}$. Put $v=(1-\chi_1)u$. It is straightforward to see that on $D$,
\begin{equation}\label{e-gue150813I}
\begin{split}
&(2\pi)\chi F_{k,\delta}(1-\chi_1)u\\
&=\sum_{\substack{(m_1,\ldots,m_d)\in\mathbb Z^d,\\ \abs{m_1\beta_1+\cdots+m_d\beta_d}\leq 2k\delta}}\int
e^{i\langle x_{2n-1}-y_{2n-1},\xi_{2n-1}\rangle}
\chi(x)\tau_\delta\Big(\frac{\xi_{2n-1}}{k}\Big)\\
&\quad\times e^{i(\sum^d_{j=1}m_j\beta_j)y_{2n-1}-im_1\theta_1-\cdots-im_d\theta_d}
v((e^{i\theta_1},\ldots,e^{i\theta_d})\circ x')dT_d\,d\xi_{2n-1}\,dy_{2n-1}.
\end{split}
\end{equation}
Let $\varepsilon>0$ be a small constant so that for every $(x_1,\ldots,x_{2n-1})\in{\rm Supp\,}\chi$, we have
\begin{equation}\label{e-gue150813II}
(x_1,\ldots,x_{2n-2},y_{2n-1})\in\set{x\in D;\, \chi_1(x)=1},\ \ \forall\abs{y_{2n-1}-x_{2n-1}}<\varepsilon.
\end{equation}
Let $\psi\in C^\infty_0((-1,1))$, $\psi=1$ on $\big[-\frac{1}{2},\frac{1}{2}\big]$. Put
\begin{equation}\label{e-gue150813III}
\begin{split}
I_0(x)=\frac{1}{(2\pi)}\sum_{\substack{(m_1,\ldots,m_d)\in\mathbb Z^d,\\ \abs{m_1\beta_1+\cdots+m_d\beta_d}\leq 2k\delta}}&\int e^{i\langle x_{2n-1}-y_{2n-1},\xi_{2n-1}\rangle}\Big(1-\psi\Big(\frac{x_{2n-1}-y_{2n-1}}{\varepsilon}\Big)\Big)\chi(x)\\
&\times \tau_\delta\Big(\frac{\xi_{2n-1}}{k}\Big)e^{i(\sum^d_{j=1}m_j\beta_j)y_{2n-1}-im_1\theta_1-\cdots-im_d\theta_d}\\
&\times v((e^{i\theta_1},\ldots,e^{i\theta_d})\circ x')dT_dd\xi_{2n-1}dy_{2n-1},
\end{split}
\end{equation}
\begin{equation}\label{e-gue150813IV}
\begin{split}
I_1(x)=\frac{1}{(2\pi)}\sum_{\substack{(m_1,\ldots,m_d)\in\mathbb Z^d}}&\int e^{i\langle x_{2n-1}-y_{2n-1},\xi_{2n-1}\rangle}\psi\Big(\frac{x_{2n-1}-y_{2n-1}}{\varepsilon}\Big)\chi(x)
\tau_\delta\Big(\frac{\xi_{2n-1}}{k}\Big)\\
&\times e^{i(\sum^d_{j=1}m_j\beta_j)y_{2n-1}-im_1\theta_1-\cdots-im_d\theta_d}\\
&\times v((e^{i\theta_1},\ldots,e^{i\theta_d})\circ x')dT_dd\xi_{2n-1}dy_{2n-1},
\end{split}
\end{equation}
and
\begin{equation}\label{e-gue150813V}
\begin{split}
I_2(x)=\frac{1}{(2\pi)}\sum_{\substack{(m_1,\ldots,m_d)\in\mathbb Z^d,\\ \abs{m_1\beta_1+\cdots+m_d\beta_d}>2k\delta}}&\int e^{i\langle x_{2n-1}-y_{2n-1},\xi_{2n-1}\rangle}\psi\Big(\frac{x_{2n-1}-y_{2n-1}}{\varepsilon}\Big)\chi(x)
\tau_\delta\Big(\frac{\xi_{2n-1}}{k}\Big)\\
&\times e^{i(\sum^d_{j=1}m_j\beta_j)y_{2n-1}-im_1\theta_1-\cdots-im_d\theta_d}\\
&\times v((e^{i\theta_1},\ldots,e^{i\theta_d})\circ x')dT_dd\xi_{2n-1}dy_{2n-1}.
\end{split}
\end{equation}
It is clear that on $D$,
\begin{equation}\label{e-gue150813VI}
\chi F_{k,\delta}(1-\chi_1)u(x)=I_0(x)+I_1(x)-I_2(x).
\end{equation}

On $D$, write $I_j(x)=s^k(x)\otimes\Td I_j(x)$, $\Td I_j(x)\in C^\infty(D)$, $j=0,1,2$.
By using integration by parts with respect to $\xi_{2n-1}$  and $y_{2n-1}$ several times and \eqref{e-gue150813}, we conclude that for every $N\gg1$ and $\ell\in\mathbb N$, there is a constant $C_{N,\ell}>0$ independent of $u$ and $k$ such that
\begin{equation}\label{e-gue150813VII}
\begin{split}
\norm{e^{-k\Phi(x)}\Td I_0(x)}_{C^\ell(D)}&\leq C_{N,\ell}k^{-2N}\sum_{\substack{(m_1,\ldots,m_d)\in\mathbb Z^d}}\Bigr(\frac{1}{m_1\beta_1+\cdots+m_d\beta_d}\Bigr)^{2N}\\
&\leq\Td C_{N,\ell}k^{-2N},
\end{split}
\end{equation}
where $\Td C_{N,\ell}>0$ is a constant independent of $u$ and $k$. Similarly, by using integration by parts with respect to $y_{2n-1}$ several times and \eqref{e-gue150813}, we conclude that for every $N>0$ and $\ell\in\mathbb N$, there is a constant $\hat C_{N,\ell}>0$ independent of $u$ and $k$ such that
\begin{equation}\label{e-gue150813VIII}
\norm{e^{-k\Phi(x)}\Td I_2(x)}_{C^\ell(D)}\leq\hat C_{N,\ell}k^{-N}.
\end{equation}
We can check that
\begin{equation}\label{e-gue150813aVIII}
\begin{split}
&I_1(x)\\
&=\frac{1}{2\pi}\int e^{i\langle x_{2n-1}-y_{2n-1},\xi_{2n-1}\rangle}\psi\Big(\frac{x_{2n-1}-y_{2n-1}}{\varepsilon}\Big)
\chi(x)\tau_\delta\Big(\frac{\xi_{2n-1}}{k}\Big)v(x',y_{2n-1})d\xi_{2n-1}dy_{2n-1}.
\end{split}
\end{equation}
From \eqref{e-gue150813II} and \eqref{e-gue150813aVIII}, we deduce that
\begin{equation}\label{e-gue150813bVIII}
\mbox{$\Td I_1(x)=0$ on $D$.}
\end{equation}
On $D$, write $\chi F_{k,\delta}(1-\chi_1)u=s^k\otimes h$, $h\in C^\infty(D)$.
From \eqref{e-gue150813VI}, \eqref{e-gue150813VII}, \eqref{e-gue150813VIII} and \eqref{e-gue150813bVIII}, we conclude that for every $N>0$ and $\ell\in\mathbb N$, there is a constant $C_{N,\ell}>0$ independent of $u$ and $k$ such that
\begin{equation}\label{e-gue150813gVIII}
\norm{e^{-k\Phi(x)}h(x)}_{C^\ell(D)}\leq\hat C_{N,\ell}k^{-N}.
\end{equation}

Let $\set{f_1,\ldots,f_{d_k}}$ be an orthonormal basis for $\mathcal{H}^0_{b,\leq k\delta}(X,L^k)$. On $D$, write
\[\chi F_{k,\delta}(1-\chi_1)f_j=s^k\otimes h_j,\ \ h_j\in C^\infty(D),\ \ j=1,2,\ldots,d_k.\]
From \eqref{e-gue150813} and \eqref{e-gue150813gVIII}, it is not difficult to see that
\begin{equation}\label{e-gue150813vgVIII}
\mbox{$\sum\limits^{d_k}_{j=1}\abs{(\pr^\alpha_xh_j)(x)e^{-k\Phi(x)}}^2=O(k^{-\infty})$ on $D$},\ \ \forall\alpha\in\mathbb N^{2n-1}_0.
\end{equation}
From \eqref{e-gue150813vgVIII}, the lemma follows.
\end{proof}

From Lemma~\ref{l-gue131209} and Lemma~\ref{l-gue150813},
we see that the conditions {\rm (I)\,}, {\rm (II)\,} and {\rm (III)\,}
in Theorem~\ref{t-gue150811} holds. The proof of Theorem~\ref{t-gue150807} is completed.

From Theorem~\ref{t-gue150807}, we can repeat the proof of Theorem 1.3 in~\cite{HLM16} (see Section 4 in~\cite{HLM16}) and get Theorem~\ref{t-gue150807I}. We omit the details.

 \bigskip

{\emph{\ \textbf{Acknowledgements.} The second author would like to
thank Professor Homare Tadano for useful discussion in this
work.  }}

\end{document}